\documentclass[a4paper,reqno]{amsart}
\usepackage{amsmath,amssymb,amsfonts,amsthm}
\usepackage{mathtools}
\usepackage{fancyhdr,fancyvrb}
\usepackage{graphicx}
\usepackage{nextpage}
\usepackage[dvipsnames,usenames]{color}
\usepackage[colorlinks=true, linkcolor=blue, citecolor=red]{hyperref}
\usepackage{tikz}
\usepackage{tikz-3dplot}
\usepackage{mathrsfs}
\usepackage{stmaryrd}

 \makeatletter
 \@namedef{subjclassname@2020}{%
 \textup{2020} Mathematics Subject Classification}
 \makeatother

\newcommand{\N}{\mathbb{N}}
\newcommand{\R}[1]{\mathbb{R}^{#1}}
\renewcommand{\S}[1]{\mathbb{S}^{#1}}

\newcommand{\Z}[1]{\mathbb{Z}^{#1}}

\newcommand{\cA}{\mathcal A}

\newcommand{\cC}{\mathcal C}

\newcommand{\cH}{\mathcal H}

\newcommand{\cL}{\mathcal L}
\newcommand{\cM}{\mathcal M}

\newcommand{\scrR}{\mathscr{R}}

\newcommand{\bulk}{\mathrm{bulk}}

\newcommand{\de}{\mathrm d}

\newcommand{\seq}{\mathrm{seq}}

\newcommand{\surface}{\mathrm{surface}}

\newcommand{\eps}{\varepsilon}

\newcommand{\norm}[1]{\left\| #1 \right\|}
\newcommand{\norma}[1]{\left\lVert#1\right\rVert}

\newcommand{\abs}[1]{\left|#1\right|}

\newcommand{\dist}[2]{\operatorname{dist}(#1,#2)}

\newcommand{\sgn}{\operatorname{sgn}}

\usepackage{soul}

\renewcommand{\geq}{\geqslant}
\renewcommand{\leq}{\leqslant}

\newcommand{\wsto}{\stackrel{*}{\rightharpoonup}}

\newcommand{\average}{{\mathchoice {\kern1ex\vcenter{\hrule
height.4pt width 8pt depth0pt}
\kern-11pt} {\kern1ex\vcenter{\hrule height.4pt width 4.3pt
depth0pt} \kern-7pt} {} {} }}
\newcommand{\ave}{\average\int}

\newcommand{\res}{\mathop{\hbox{\vrule height 7pt width .5pt depth
0pt\vrule height .5pt width 6pt depth0pt}}\nolimits}

\mathchardef\emptyset="001F

\numberwithin{equation}{section}

\newtheorem{defin}{Definition}[section]
\newtheorem{remark}[defin]{Remark}
\newtheorem{theorem}[defin]{Theorem}
\newtheorem{lemma}[defin]{Lemma}
\newtheorem{proposition}[defin]{Proposition}

\title{Measure structured deformations}%
\author[S.~Kr\"{o}mer]{Stefan Kr\"{o}mer}
\address[S.~Kr\"{o}mer]{Institute of Information Theory and Automation, Czech Academy of Sciences, Pod vod\'{a}renskou ve\v{z}\'{i} 4, CZ-182 00, Prague 8, Czech Republic}
\email{skroemer@utia.cas.cz}
\author[M.~Kru\v{z}\'{i}k]{Martin Kru\v{z}\'{i}k}
\address[M.~Kru\v{z}\'{i}k]{Institute of Information Theory and Automation, Czech Academy of Sciences, Pod vod\'{a}renskou ve\v{z}\'{i} 4, CZ-182 00, Prague 8, Czech Republic}
\email{kruzik@utia.cas.cz}
\author[M.~Morandotti]{Marco Morandotti}
\address[M.~Morandotti]{Dipartimento di Scienze Matematiche ``G.~L.~Lagrange'', Politecnico di Torino, Corso Duca degli Abruzzi, 24, 10129 Torino, Italy}
\email{marco.morandotti@polito.it}
\author[E.~Zappale]{Elvira Zappale}
\address[E.~Zappale]{Dipartimento di Scienze di Base e Applicate per l'Ingegneria, Sapienza Universit\`{a} di Roma, Via Antonio Scarpa, 16, 00161 Roma, Italy}
\email{elvira.zappale@uniroma1.it}

\date{\today}


\makeindex

\begin{document}

\begin{abstract}
Measure structured deformations are introduced to present a unified theory of deformations of continua. 
The energy associated with a measure structured deformation is defined via relaxation departing either from energies associated with classical deformations or from energies associated with structured deformations.
A concise integral representation of the energy functional is provided both in the unconstrained case and under Dirichlet conditions on a part of the boundary.
\end{abstract}

\maketitle

\tableofcontents



\section{Introduction} \label{sec:intro}
The primary objective of continuum mechanics in solids is to articulate how a solid body will alter its shape when subjected to specified external forces or boundary conditions. 
A crucial initial step towards achieving this objective involves selecting a category of deformations for the continuum. 
In describing numerous continua, certain widely accepted criteria for the chosen category of deformations have been established: these deformations should be reversible, with differentiable mappings and inverses, and the combination of two deformations within this category should result in another deformation within the same category. 
However, classical deformations may not always suffice for describing all continua, requiring alternative selections in many cases. 
One approach involves introducing additional kinematic variables, such as the director fields in a polar continuum. An alternative approach entails incorporating supplementary fields that, while connected to the deformation, function as internal variables. 
For instance, in theories concerning plasticity, the plastic deformation tensor follows an evolutionary law outlined in the constitutive equations of the continuum. 

Del Piero and Owen \cite{PieOwe93} proposed 
an alternative approach that identifies 
classes of deformations called \emph{structured deformations}, suited for continua featuring supplementary kinematical variables, as well as for continua featuring internal variables (we refer the reader to \cite{MMObook2023} for a comprehensive survey 
on this topic).
In the theory of structured deformations, if $\Omega\subset\R{N}$ is the continuum body, the role usually played by the deformation field~$u\colon\Omega\to\R{d}$ and by its gradient $\nabla u\colon\Omega\to\R{d\times N}$ is now played by a triple $(\kappa,g,G)$, where the piecewise differentiable field  $g\colon\Omega\setminus\kappa\to\R{d}$ is the macroscopic deformation and the piecewise continuous matrix-valued field $G\colon\Omega\setminus\kappa\to\R{d\times N}$ captures the contribution at the macroscopic level of smooth submacroscopic changes.
The (possibly empty) discontinuity set $\kappa\subset \Omega$ of~$g$ and~$G$ can be regarded as the crack set of the material.
The main result obtained by Del Piero and Owen is the Approximation Theorem \cite[Theorem~5.8]{PieOwe93} stating that any structured deformation $(\kappa,g, G)$ can be approximated (in the $L^\infty$ convergence) by a sequence of \emph{simple deformations} $\{(\kappa_n,u_n)\}$.
The matrix-valued field $\nabla g-G$ captures the effects of submacroscopic \emph{disarrangements}, which are slips and separations that occur at the submacroscopic level.
The spirit with which structured deformations were introduced was that of enriching the existing class of energies suitable for the variational treatment of physical phenomena without having to commit at the outset to a specific mechanical theory such as elasticity, plasticity, or fracture. 
Ideally, the regime of the deformation is energetically chosen by the body depending on the applied external loads: if these are small, then the deformation will most likely be elastic, whereas if these are large, a plastic regime or even fracture may occur.

The natural mathematical context to study problems similar to those mentioned above is that of calculus of variations, in which equilibrium configurations of a deforming body subject to external forces are obtained as minimizers of a suitable energy functional.
In the classical theories where the mechanics is described by the gradient of the deformation field~$u$, a typical expression of the energy is
\begin{equation}\label{001intro}
E(\kappa,u;\Omega)\coloneqq \int_\Omega W(\nabla u)\,\de x+\int_{\Omega\cap\kappa} \psi([u],\nu_u)\,\de \cH^{N-1},
\end{equation}
where $W\colon \R{d\times N}\to[0,+\infty)$ and $\psi\colon\R{d}\times\S{N-1}\to[0,+\infty)$ are continuous functions satisfying suitable structural assumptions and model the bulk and interfacial energy densities, respectively.
In the context of Del Piero and Owen, it is not clear how to assign energy to a structured deformation $(\kappa,g,G)$;  
the issue was solved by Choksi and Fonseca who, providing a suitable version of the Approximation Theorem \cite[Theorem~2.12]{ChoFo97}, use the technique of relaxation to assign the energy $I(g,G;\Omega)$ as the minimal energy along sequences $\{u_n\}\subset SBV(\Omega;\R{d})$ converging to $(g,G)\in SBV(\Omega;\R{d})\times L^1(\Omega;\R{d\times N})\eqqcolon SD(\Omega;\R{d}\times\R{d\times N})$ in the following sense:
\begin{equation}\label{SBV_conv}
u_n\wsto g\quad\text{in $BV(\Omega;\R{d})$}\qquad\text{and}\qquad  \nabla u_n\wsto G \quad\text{in $\cM(\Omega;\R{d\times N})$},
\end{equation}
where $\nabla u_n$ denotes the absolutely continuous part of the distributional gradient~$Du$.
More precisely, the relaxation process reads
\begin{equation}\label{energyI}
I(g,G;\Omega)\coloneqq \inf_{\{u_n\}}\Big\{ \liminf_{n\to\infty} E(S_{u_n},u_n;\Omega): \text{$u_n\to(g,G)$ according to \eqref{SBV_conv}}\Big\}
\end{equation}
and is accompanied by integral representation theorems in $SD(\Omega;\R{d}\times\R{d\times N})$ for the relaxed energy $I(g,G;\Omega)$ (see \cite[Theorems~2.16 and~2.17]{ChoFo97} and \cite[Theorem~3]{OP2015}).
The reader might have noticed that the crack set~$\kappa$ has been identified with the jump set~$S_{u_n}$ of the field $u_n\in SBV(\Omega;\R{d})$.
The variational setting introduced in \cite{ChoFo97} gave rise to numerous applications of structured deformations in various contexts, see \cite{AMMZ,BMMOexp,BaMaMoOwZa22a,BMZ,CMMO,MMOZ,MMZ}, in which an explicit form of the energy $I(g,G;\Omega)$ could be provided. 

We stress that, although we look at targets $(g,G)$ belonging to $SBV(\Omega;\R{d})\times L^1(\Omega;\R{d\times N})$, in general, the convergence \eqref{SBV_conv} might lead to limits that are in $BV(\Omega;\R{d})\times \cM(\Omega;\R{d\times N})$ and that, in assigning the energy \eqref{energyI}, Choksi and Fonseca make the explicit choice to represent the relaxed energy only in $SBV(\Omega;\R{d})\times L^1(\Omega;\R{d\times N})$.
Moreover, from the mechanical point of view, one cannot, in principle, exclude that $\{\nabla u_n\}$ develop singularities in the limit, which would reflect on a weaker regularity of the field~$G$, possibly not even of the same type as those of the singular part~$D^sg$ of the distributional derivative~$Dg$, as is the case in \cite{BMS,BMMO,MS}.
Both these mathematical and mechanical reasons suggest that the definition of structured deformations should be extended from $SD(\Omega;\R{d}\times\R{d\times N})$ to the larger set 
\begin{equation}\label{mSD}
mSD(\Omega;\R{d}\times\R{d\times N})\coloneqq BV(\Omega;\R{d})\times \cM(\Omega;\R{d\times N}),
\end{equation}
which we call \emph{measure structured deformations}, and which we abbreviate here with $mSD$.

In this paper, we generalize the results of \cite{ChoFo97} to $mSD$.
In particular, denoting with $u_n\wsto(g,G)$ in $mSD$ the convergence in~\eqref{SBV_conv}, we prove the Approximation Theorem~\ref{thm:approx}: given any measure structured deformation $(g,G)\in mSD$, there exists a sequence $\{u_n\}\subset SBV(\Omega;\R{d})$ such that $u_n\wsto(g,G)$ in $mSD$.
This serves to define the energy $I\colon mSD\to[0,+\infty)$ via the relaxation \eqref{energyI} in the larger space $mSD$, see \eqref{defI}, for which we prove the integral representation result, Theorem~\ref{thm:representation}.
This is one of the main results of the paper, in which we recover the same structure of \cite[Theorems~2.16 and~2.17]{ChoFo97} and \cite[Theorem~3]{OP2015}, with the presence of an additional diffuse part.
One of the novelties of our setting is that we manage to obtain a concise form of the relaxed energy functional involving only a bulk contribution~$H$ and its recession function at infinity~$H^\infty$
$$I(g,G;\Omega)=\int_\Omega H\Big(\nabla g,\frac{\de G}{\de\cL^N}\Big)\,\de x +\int_\Omega H^\infty\Big(\frac{\de(D^sg,G^s)}{\de |(D^sg,G^s)|}\Big)\,\de |(D^sg,G^s)|(x),$$ 
where~$D^sg$ and~$G^s$ are the singular parts of the measures~$Dg$ and~$G$, respectively,
see \eqref{elegant}, in the typical form of Goffman and Serrin \cite{GS64} for functionals defined on measures for a density, which is a particular case of those treated in \cite{ArDPRi20a}. 
The relaxed bulk energy density~$H$ turns out to be quasiconvex-convex; see Proposition~\ref{prop:H-ws-lsc}. 
It is interesting to notice that not every quasiconvex-convex function can be obtained as the bulk energy density associated with a structured deformation: ours retains the memory of the specific relaxation process~\eqref{defI} (see also the counterexample in Proposition~\ref{counterexample}).
In Theorem~\ref{thm:I-alternative} we prove that the energy $I(g,G;\Omega)$ can be obtained by relaxing from $SD(\Omega;\R{d}\times\R{d\times N})$ to $mSD(\Omega;\R{d}\times\R{d\times N})$ the energy \eqref{001intro} with the addition of a term penalizing the structuredness $\nabla g-G$
$$\hat{E}_R(g,G;\Omega)\coloneqq E(S_g,g;\Omega)+R\int_\Omega |\nabla g-G|\,\de x,$$
see \eqref{E_R}. 
Another relevant result is the possibility of performing the relaxation under trace constraints, see Theorem~\ref{lbclamped}, which has the far-reaching potential of studying minimization problems in $mSD(\Omega;\R{d}\times\R{d\times N})$ with the addition of boundary data.

\section{Setting and the definition of the energy in $mSD$}\label{sec:results}
We assume that the main results about functions of bounded variations are known, otherwise we refer the reader to the monograph \cite{AFP} for a thorough introduction; likewise, we refer the reader to \cite{DM93} for an introduction to relaxation (see also \cite{Bra}).

We consider an initial energy as in \eqref{001intro}, which, since we take $\kappa=S_u$, now can be written as $E\colon SBV(\Omega;\R{d})\to [0,+\infty)$
\begin{equation}\label{001}
E(u;\Omega)\coloneqq \int_\Omega W(\nabla u)\,\de x+\int_{\Omega\cap S_u} \psi([u],\nu_u)\,\de \cH^{N-1}(x),
\end{equation}
where $W\colon \R{d\times N}\to[0,+\infty)$ and $\psi\colon\R{d}\times\S{N-1}\to[0,+\infty)$ are continuous functions satisfying the following assumptions
for $A\in \R{d\times N}$, $\lambda,\lambda_1,\lambda_2\in \R{d}$ and $\nu\in\S{N-1}$:
\begin{align*}
	&c_W |A|\leq W(A)\leq C_W(1+|A|); \label{W1}\tag{$W$:1} \\
	&\text{$W$ is globally Lipschitz continuous}; \label{W2}\tag{$W$:2} \\
	&\begin{aligned}
		&\text{there exist $c>0$ and $0<\alpha<1$ such that}\\
		&\qquad \bigg|W^\infty(A)-\frac{W(tA)}{t}\bigg|\leq \frac{c\abs{A}^{1-\alpha}}{t^\alpha}
		\quad\text{whenever $t>0$ and $t\abs{A}\geq 1$,}\\ 
		&\text{where}~~W^\infty(A)\coloneqq\limsup_{t\to+\infty}\frac{W(tA)}{t};
	\end{aligned} \label{W3}\tag{$W$:3} \\
	&\text{$c_\psi|\lambda|\leq \psi(\lambda,\nu)\leq C_\psi|\lambda|$; 
	} \label{psi1}\tag{$\psi$:1} \\
	&\text{$\psi(t\lambda,\nu)=t\psi(\lambda,\nu)$ and $\psi(-\lambda,-\mu)=\psi(\lambda,\mu)$;
	} \label{psi2}\tag{$\psi$:2} \\
	&\text{$\psi(\lambda_1+\lambda_2,\nu)\leq\psi(\lambda_1,\nu)+\psi(\lambda_2,\nu)$.
	} \label{psi3}\tag{$\psi$:3}
\end{align*}

We consider \emph{measure structured deformations}, that is, pairs $(g,G)\in mSD$, see \eqref{mSD};
we endow the space $mSD$ with the norm
$$\norma{(g,G)}_{mSD}\coloneqq \norma{g}_{BV(\Omega;\R{d})}+|G|(\Omega),$$
the latter term denoting the total variation of the measure~$G$.
We are interested in assigning an energy $I\colon mSD\to[0,+\infty)$ by means of the relaxation 
\begin{equation}\label{defI}
I(g,G;\Omega)\coloneqq \inf\Big\{\liminf_{n\to\infty} E(u_n;\Omega): \{u_n\}\in\scrR(g,G;\Omega)\Big\},
\end{equation}
where, for every open set $U\subset\Omega$,
\begin{equation}\label{defR}
\scrR(g,G;U)\coloneqq\big\{\{u_n\}\subset SBV(U;\R{d}): \text{$u_n\wsto(g|_U,G|_U)$ as in \eqref{SBV_conv}}\big\}
\end{equation}
is the set of admissible sequences. 
Our main result is a representation theorem for this energy, namely that $I=J$ with 
the explicit representation of the limit functional given by 
\begin{equation}\label{defJ}
\begin{split}
\!\!\!\! J(g,G;\Omega)\coloneqq& \int_\Omega H(\nabla g,G^a)
\,\de x + \int_{\Omega\cap S_g} h^j\bigg([g], \frac{\de G^j_{g}}{\de(\cH^{N-1}\res{S_g})},\nu_g\bigg)\,\de \cH^{N-1}(x)\\
&\, + \int_{\Omega} h^c\bigg(\frac{\de D^c g}{\de|D^c g|}, \frac{\de G^c_g}{\de|D^cg|}\bigg)\,\de|D^cg|(x) + 
\int_{\Omega} h^c\bigg(0, \frac{\de G^s_g}{\de|G^s_g|}\bigg)\,\de|G^s_g|(x),
\end{split}
\end{equation}
where $H\colon\R{d\times N}\times \R{d\times N}\to[0,+\infty)$, $h^j\colon\R{d}\times\R{d\times N}\times\S{N-1}\to[0,+\infty)$, and $h^c\colon\R{d\times N}\times\R{d\times N}\to[0,+\infty)$ are suitable bulk, surface, and Cantor-type relaxed energy densities.
In \eqref{defJ}, we have the following objects: since $g\in BV(\Omega;\R{d})$, we know that, by De Giorgi's structure theorem, 
$$Dg=D^ag+D^sg=D^ag+D^jg+D^cg=\nabla g\cL^N+[g]\otimes\nu_g\cH^{N-1}\res S_g+D^cg,$$
and we can decompose
$$G=G^a+G^s=G^a+G^j_g+G^c_g+G^s_g,$$
where 
$$G^a\ll\cL^N,\quad  \de G^j_g= \frac{\de G}{\de |D^j g|}
\de |D^j g|,\quad \de G^c_g=\frac{\de G}{\de|D^cg|}\de |D^cg|,\quad G_g^s\coloneqq G-G^a-G^j_g-G^c_g\,.$$
Here, in case of $G^a$ and other measures absolutely continuous with respect to the Lebesgue measure, our notation does not distinguish between the measure and its density with respect to~$\cL^N$.
Also, notice that $G_g^s$ is singular with respect to $\cL^N+|Dg|$.

To carry out our program, we will use the following results.
\begin{theorem}[{Alberti \cite[Theorem~3]{Alberti1991}, \cite[Theorem~2.8]{ChoFo97}}]\label{Al}
Let $G \in L^1(\Omega; \R{d{\times} N})$. 
Then there exist a function $f \in SBV(\Omega; \R d)$, a Borel function $\beta\colon\Omega\to\R{d{\times} N}$, and a constant $C_N>0$ depending only on $N$ such that
\begin{equation}\label{817}
Df = G \,{\cL}^N + \beta \cH^{N-1}\res S_f, \qquad
\int_{\Omega\cap S_f} |\beta| \, \de \cH^{N-1}(x) \leq C_N \lVert G\rVert_{L^1(\Omega; \R{d {\times} N})}.
\end{equation}
\end{theorem}
\begin{lemma}[{\cite[Lemma~2.9]{ChoFo97}}]\label{ctap}
Let $u \in BV(\Omega; \R d)$. Then there exist piecewise constant functions $\bar u_n\in SBV(\Omega;\R d)$  such that $\bar u_n \to u$ in $L^1(\Omega; \R d)$ and
\begin{equation}\label{818}
|Du|(\Omega) = \lim_{n\to \infty}| D\bar u_n|(\Omega) = \lim_{n\to \infty} \int_{\Omega\cap S_{\bar u_n}} |[\bar u_n]|\; \de\cH^{N-1}(x).
\end{equation}
\end{lemma}

The following approximation theorem generalizes the one obtained in \cite{Silhavy2015}.
\begin{theorem}[approximation theorem]\label{thm:approx}
Let $\Omega\subset\R{N}$ be a bounded, open set with Lipschitz boundary.
For each $(g,G)\in mSD$ there exists a sequence $\{u_n\}\subset SBV(\Omega;\R{d})$ such that $u_n\wsto(g,G)$ in $mSD$ according to \eqref{SBV_conv}.
In addition, we have that
\begin{subequations}\label{006}
\begin{equation}\label{006_est}
\norma{Du_n}_{\cM(\Omega;\R{d\times N})}\leq C_1\norma{(g,G)}_{mSD},
\end{equation}
and
\begin{equation}\label{006_est2}
\norma{u_n}_{BV(\Omega;\R{d})}\leq C_2(\Omega)\norma{(g,G)}_{mSD},
\end{equation}
\end{subequations}
for constants $C_1=C_1(N)>0$ and $C_2(\Omega)=C_2(N,\Omega)>0$ 
independent of $\{u_n\}$ and $(g,G)$.
\end{theorem}
\begin{proof}
Let $\{G^k\}\subset L^1(\Omega;\R{d\times N})$ be a sequence of functions such that $G^k\wsto G$ as $k\to\infty$ and $\sup_{k\in\N}\norma{G^k}_{L^1(\Omega;\R{d\times N})}\leq \abs{G}(\Omega)$ (see \cite{KKZ}), and consider the corresponding pairs $(g,G^k)\in BV(\Omega;\R{d})\times L^1(\Omega;\R{d\times N})$.
By Theorem~\ref{Al}, for each $k\in\N$, there exists $f^k\in SBV(\Omega;\R{d})$ such that $\nabla f^k=G^k$ and, by the estimate in~\eqref{817}, 
\begin{subequations}\label{int_est}
\begin{equation}\label{int_est1}
\abs{Df^k}(\Omega)\leq C_N\norma{G^k}_{L^1(\Omega;\R{d\times N})}\leq C_N\abs{G}(\Omega).
\end{equation}
Since~$\Omega$ is a Lipschitz set, we can use the Poincar\'{e} inequality and obtain that 
\begin{equation}\label{int_est2}
\sup_{k\in\N}\norma{f^k}_{BV(\Omega;\R{d})}\leq C_P(\Omega) C_N\abs{G}(\Omega).
\end{equation}
\end{subequations}
By Lemma~\ref{ctap}, for each $k\in\N$ there exists a sequence $\{\bar v^k_n\}\subset SBV(\Omega;\R{d})$ of piecewise constant functions such that $\bar v^k_n\to g-f^k$ in $L^1(\Omega;\R{d})$ and, by~\eqref{818}, $\abs{D\bar v^k_n}(\Omega)\to \abs{D(g-f^k)}(\Omega)$ as $n\to\infty$.
Now, the sequence of functions $v^k_n\coloneqq \bar v^k_n+f^k$ is such that $v^k_n\to g$ in $L^1$ and $\nabla v^k_n=G^k$, as $n\to\infty$, for every $k\in\N$.
The convergences in \eqref{006} and the estimates in \eqref{006} now follow from estimates~\eqref{int_est} by a diagonal argument, by defining $u_n\coloneqq v^{k(n)}_n$, with $k(n)\to\infty$ slowly enough.
\end{proof}

Before stating our integral representation result, we define the following classes of competitors for the characterization of the relaxed energy densities below.
We let $Q\subset\R{N}$ be the unit cube centered at the origin with faces perpendicular to the coordinate axes, and for $\nu\in\S{N-1}$, we let $Q_\nu\subset\R{N}$ be the rotated unit cube so that two faces are perpendicular to~$\nu$.
For $A,B,\Lambda\in\R{d\times N}$ and $\lambda\in\R{d}$, we define
\begin{subequations}\label{competitors}
\begin{eqnarray}
\cC^{\bulk}(A,B;Q) & \!\!\!\! \coloneqq & \!\!\!\! \bigg\{u\in SBV(Q;\R{d}):u|_{\partial Q}(x)=(Ax)|_{\partial Q},\ave_{Q} \nabla u\,\de x=B\bigg\}, \\
\!\!\!\! \!\!\!\! \cC^{\surface}(\lambda,\Lambda;Q_\nu) & \!\!\!\! \coloneqq & \!\!\!\! \bigg\{u\in SBV(Q_\nu;\R{d}): u|_{\partial Q_\nu}(x)= s_{\lambda,\nu}|_{\partial Q_\nu}(x), \ave_{Q_\nu} \nabla u\,\de x=\Lambda\bigg\},\quad\quad
\end{eqnarray}
\end{subequations}
where $s_{\lambda,\nu}(x)\coloneqq \frac12\lambda(\sgn(x\cdot\nu)+1)$.
Moreover,  for any open set $U\subset \R{N}$ and $v\in SBV(U;\R{d})$, we let
\begin{equation}\label{def-Einfty}
E^\infty(v;U)\coloneqq \int_{U} W^\infty(\nabla v)\,\de x+\int_{U\cap S_v} \psi([v],\nu_v)\,\de\cH^{N-1}(x).
\end{equation}

\begin{theorem}[integral representation]\label{thm:representation}
Let $\Omega\subset \R{N}$ be a bounded Lipschitz domain, and assume that \eqref{W1}--\eqref{W3} and \eqref{psi1}--\eqref{psi3} hold true.
Then 
\[
	I(g,G;\Omega)=J(g,G;\Omega)\quad\text{for all }(g,G)\in mSD(\Omega;\R{d}\times\R{d\times N}),
\]
where $I$ and $J$ are defined in \eqref{defI} and \eqref{defJ}, respectively, and the densities in $J$ are given by
\begin{subequations}\label{relaxed_energy_densities}
\begin{eqnarray}
H(A,B) & \!\!\!\! \coloneqq & \!\!\!\!  
\inf\big\{E(u;Q): u\in \cC^{\bulk}(A,B;Q)\big\}; \label{def-H} \\
h^j(\lambda,\Lambda,\nu) & \!\!\!\! \coloneqq & \!\!\!\! 
\inf\big\{E^{\infty}(u;Q_\nu): u\in \cC^{\surface}(\lambda,\Lambda;Q_\nu)\big\}; \label{def-hj}\\
h^c(A,B) &\!\!\!\! \coloneqq & \!\!\!\!  
\inf\big\{E^{\infty}(u;Q): u\in \cC^{\bulk}(A,B;Q)\big\}. \label{def-hc} 
\end{eqnarray}
\end{subequations}
\end{theorem}
The proof is given in Section~\ref{sec:upper} (upper bound: $I\leq J$) and Section~\ref{sec:lower bound} (lower bound: $I\geq J$).

\begin{remark}\label{rem:ChoFo}
For the special case $(g,G)\in SD(\Omega;\R{d}\times\R{d\times N})$, Theorem~\ref{thm:representation} reduces to
\cite[Theorem~2.16]{ChoFo97} (for the functional $I_1$ in the notation of \cite{ChoFo97}). Unlike \cite{ChoFo97}, we assumed coercivity of~$W$ in \eqref{W1}, but only to avoid additional technicalities.
\end{remark}
\begin{remark}\label{rem:J_as_int_dH}
As shown in Proposition~\ref{prop:H-hj-hc} below, 
$h^c$ coincides with the recession function $H^\infty$ of $H$, and 
$h^j$ can be replaced by $h^c=H^\infty$, more precisely,
\begin{equation}\label{103}
h^j(\lambda,\Lambda,\nu)=h^c(\lambda\otimes\nu,\Lambda).
\end{equation}  
This allows for another, much more elegant representation of $J$: 
\begin{equation}\label{elegant}
	J(g,G)= \! \int_\Omega \de H(Dg,G)
	= \! \int_\Omega \! H\Big(\nabla g,\frac{\de G}{\de\cL^N}\Big)\,\de x
	+ \! \int_\Omega \! H^\infty\Big(\frac{\de(D^sg,G^s)}{\de |(D^sg,G^s)|}\Big)\,\de |(D^sg,G^s)|(x), 
\end{equation}
see Definition~\ref{def:dh} and Proposition~\ref{prop:J-with-H-only} below.
\end{remark}

\begin{remark}\label{remass} 
\begin{enumerate}
\item[(i)]
As a consequence of \eqref{psi1} and \eqref{psi3}, $\psi$ is also globally Lipschitz in $\lambda$:
\begin{equation}\label{psiLip}
	\abs{\psi(\lambda_1,\nu)-\psi(\lambda_2,\nu)}\leq C_\psi \abs{\lambda_1-\lambda_2}.
\end{equation}
\item[(ii)]
We will never use the symmetry condition in \eqref{psi2} directly, but it is necessary to
make~$E$ well-defined in~$SBV$, as jump direction and  jump normal are only uniquely 
defined up to a simultaneous change of sign.
\end{enumerate}
\end{remark}

\begin{remark}[Instability of the contribution of $G^s_g$ in $I=J$]\label{rem:instability}
As in the case of typical integral functionals in $BV$ with $G=0$, the individual contributions
in $J$ handling each of the four components of the measure decomposition
\[
	(Dg,G)
	=\frac{\de (Dg,G)}{\de \cL^N}\cL^N+\frac{\de (Dg,G)}{\de |D^a g|}|D^a g|
	+\frac{\de (Dg,G)}{\de |D^j g|}|D^j g|+\frac{\de (Dg,G)}{\de |G_g^s|}|G_g^s|
\]
are not continuous with respect to strict or area-strict convergence; for instance, Lebesgue-absolutely continuous
contributions can generate Cantor or jump contributions in the limit. 
The last contribution in $J$ 
of the singular rest $G_g^s$  is even worse than the others, though, because it 
is not even continuous in the norm topology of $BV(\Omega;\R{d})\times \cM(\Omega;\R{d\times M})$.

Take, for instance, $N=1$, 
\[
  \Omega\coloneqq(-1,1),~~W\coloneqq\abs{\cdot},~~\psi(\cdot,\nu)\coloneqq\abs{\cdot},~~
	g_k\coloneqq\frac{1}{k}\chi_{(0,1)},~~G\coloneqq\delta_0.
\]
In particular, $h^c(0,B)=\abs{B}$ for all $B\in \R{}$.
Then $(g_k,G)\to (g,G)=(0,\delta_0)$ strongly in $BV\times \cM$, but
$G_{g_k}^s=0$ for all $k$ while $G_{g}^s=\delta_0$ (since $D^jg_k=\frac{1}{k} \delta_0$,  
the whole singular contribution of $G$ with respect to $\cL^N+|Dg_k|$ is captured by $\frac{\de G}{\de |D^jg_k|}|D^jg_k|=k\frac{1}{k}\delta_0=\delta_0$,  
while $g=0$ so that $G^s_g=\delta_0=G$). 
As a consequence, the contribution of $G_{g_k}^s$ in $J$ jumps in the limit as $k\to\infty$:
\[
	\lim_{k\to\infty} \int_\Omega h^c\Big(0,\frac{\de G_{g_k}^s}{\de |G_{g_k}^s|}\Big)\,\de |G_{g_k}^s|(x)=0\neq 1=h^c(0,1)= 
	\int_\Omega h^c\Big(0,\frac{\de G_{g}^s}{\de |G_{g}^s|}\Big)\,\de |G_{g}^s|(x).
\]
\end{remark}

%

\section{Auxiliary results}
In this section, we present some auxiliary results that are pivotal for the proof of Theorem~\ref{thm:representation}.
In particular, we show that all three densities~$H$,~$h^j$, and~$h^c$ are linked (Proposition~\ref{prop:H-hj-hc}) and we present a sequential characterization for them (Proposition~\ref{prop:seq-densities}).
In Section~\ref{sec32}, functionals depending on measures are introduced, as well as the notion of area-strict convergence.

\subsection{Equivalent characterizations of the relaxed energy densities}
\begin{proposition}\label{prop:H-hj-hc}
Assume that \eqref{W3} and \eqref{psi2} hold true and $H$, $h^j$, and $h^c$ are 
defined as in Theorem~\ref{thm:representation}. Then
 the strong recession function of $H$,
\[
	H^\infty(A,B)\coloneqq\lim_{t\to+\infty} \frac{H(tA,tB)}{t},
\]
exists. Moreover, we have that 
\begin{equation}\label{hc-Hinfty}
	h^c=H^\infty 
\end{equation}
and for all $B\in \R{d\times N}$, $\lambda\in \R{d}$, and $\nu\in \S{N-1}$, 
\begin{equation}\label{hj-hc}
	h^c(\lambda\otimes \nu,B)=h^c_\nu(\lambda\otimes \nu,B)=h^j(\lambda,B,\nu),
\end{equation}
where $h^c_\nu$ is obtained from $h^c$ by replacing the standard unit cube $Q$ by the unit cube $Q_\nu$ oriented according to the normal $\nu$, i.e.,
\begin{equation}\label{def-hcnu}
	h^c_\nu(A,B)\coloneqq 
\inf\big\{E^{\infty}(u;Q_\nu): u\in \cC^{\bulk}(A,B;Q_\nu)\big\}.
\end{equation}
\end{proposition}
\begin{proof}
We define $H^\infty(A,B)\coloneqq \limsup_{t\to \infty}\frac{1}{t}H(tA,tB)$.
With this definition, we obtain \eqref{hc-Hinfty} as a consequence of \eqref{W3} and \eqref{psi2}.
Moreover, this even holds if $t$ is replaced by an arbitrary subsequence. The $\limsup$ above is thus independent of subsequences and, therefore, a limit.
It remains to show \eqref{hj-hc}.

{\bf First equality in \eqref{hj-hc}:} We claim that $h^c(A,B)=h^c_\nu(A,B)$ for arbitrary $A,B\in \R{d\times N}$. We will first show that $h^c_\nu(A,B)\leq h^c(A,B)$. 
Let $\eps>0$ and choose an $\eps$-almost minimizer $u\in SBV(Q;\R{d})$ for the infimum in the definition of $h^c(A,B)$:
\begin{equation}\label{hcnu-01}
	E(u;Q)\leq h^c_\nu(A,B)+\eps,~~~u=Ax\text{ on }\partial {Q},~~\ave_{Q} \nabla u\,\de x=B.
\end{equation}
Up to a set of measure zero, $Q_\nu$ can be covered with countably many shifted and rescaled, pairwise disjoint copies of $Q$:
\begin{equation} \label{hcnu-02}
	\bigcup_{i\in \N} x_i+\delta_i Q\subset Q_\nu \subset \bigcup_{i\in \N} (x_i+\delta_i \overline{Q}),
\end{equation}
with  suitable $x_i\in Q$, $0<\delta_i\leq 1$.
Defining
\begin{equation}
	\tilde{u}(x)\coloneqq \sum_i \chi_{x_i+\delta_i Q}(x) \left(Ax_i+\delta u\Big(\frac{x-x_i}{\delta}\Big)\right),
\end{equation}
we obtain $\tilde{u}\in SBV(Q_\nu;\R{d})$ with
\begin{equation} \label{hcnu-04}
	|D\tilde{u}|(x_i+\delta_i \partial Q)=0\quad\text{for all $i\in \N$} 
\end{equation}
and
$\tilde{u}=Ax$ on $\partial {Q_\nu}$ (as well as on $x_i+\delta_i \partial Q$).
Observe that by the definition of $E^\infty$ in \eqref{def-Einfty}, the positive one-homogeneity of $W^\infty$ and $\psi$ and a change of variables,
\begin{equation} \label{hcnu-05}
	E^\infty(\tilde{u};x_i+\delta_i Q)=\delta_i^N E^\infty(u;x_i+\delta_i Q)
	\leq \delta_i^N\left(h^c_\nu(A,B)+\eps\right),
\end{equation}
the latter due to \eqref{hcnu-01}. 
In addition, \eqref{hcnu-02} gives that
$\sum_{i\in\N} \delta_i^N=\sum_{i\in\N} \cL^N(x_i+\delta_i Q)=\cL^N(Q_\nu)=1$.
Using the additivity of the integrals in $E$, \eqref{hcnu-02} and \eqref{hcnu-04}, we can sum \eqref{hcnu-05} over $i$ to conclude that
\begin{equation}
	E^\infty(\tilde{u};Q_\nu)\leq 
	h^c_\nu(A,B)+\eps.
\end{equation}
Similarly, we can also check that $\ave_{Q_\nu} \nabla \tilde{u}\,\de x=B$.
Since $\eps>0$ was arbitrary and $\tilde{u}$ is admissible for the infimum in the definition of $h^c_\nu$, this 
implies that $h^c_\nu(A,B)\leq h^c(A,B)$. The opposite inequality follows in exactly the same way, with exchanged roles of $Q$ and $Q_\nu$.
 
{\bf Second equality in \eqref{hj-hc}:} We have to show that $h^c_\nu(\lambda\otimes \nu,B)=h^j(\lambda,B,\nu)$.
For $k\in \N$, define the laterally extended cuboid
\[
\begin{aligned}
	R_\nu(k)\coloneqq \Big\{x\in \R{N}\,\Big|\, \abs{x\cdot \nu}<\frac{1}{2},~~
	\abs{x\cdot \nu_j^\perp}<\frac{2k+1}{2}~~\text{for}~j=1,\ldots,N-1
	\Big\},&
\end{aligned}
\]
where $\nu_j^\perp$, $j=1,\ldots,N-1$, are the pairwise orthogonal unit vectors perpendicular to $\nu$
corresponding to the lateral faces of $Q_\nu$.
Notice that up to a set of measure zero formed by overlapping boundaries, $R_\nu(k)$ can be written as a pairwise disjoint union of $(2k+1)^{N-1}$ shifted copies of~$Q_\nu$:
\[
	\overline{R}_\nu(k)=\bigcup_{\xi\in Z(k)}(\xi+ \overline{Q}_\nu),~~~Z(k)\coloneqq 
	\left\{\left. \xi=\sum_{i=1}^{N-1} j(i) \nu_{j(i)}^\perp \,\right|\,j(i)\in \{-k,\ldots,k\} \right\}.
\]
Now let $\eps>0$ and choose an $\eps$-almost minimizer $u\in SBV(Q_\nu;\R{d})$ for the infimum in the definition of $h^c_\nu(\lambda\otimes\nu,B)$:
\begin{equation}\label{eq:hc-hj-1}
	h^c_\nu(\lambda\otimes \nu,B)+\eps \geq 
	E^\infty\big(u;Q_\nu\big),
\end{equation}
with $E^\infty$ defined in \eqref{def-Einfty}.
Since $\nu_j^\perp\cdot \nu=0$, the affine function $x\mapsto (\lambda\otimes\nu)x$ determining the boundary values of $u$ is constant
direction $\nu_j^\perp$ for each $j=1,\ldots,N-1$.
We can therefore extend $u$ periodically in the $(N-1)$ directions $\nu_j^\perp$ to a function 
$u_k\in SBV(R_\nu(k);\R{d})$, without creating jumps at the interfaces between elementary cells of periodicity: 
$u_k|_{Q_\nu}=u$, 
$u_k(x+\nu_j^\perp)=u_k(x)$ whenever $x,x+\nu_j^\perp\in R_\nu(k)$,
$u_k=(\lambda\otimes\nu)x$ on $\xi+ \partial Q_\nu$ for each $\xi\in Z(k)$ (in the sense of traces), and
$|Du_k|(\xi+ \partial Q_\nu)=0$ for each $\xi\in Z(k)$.
As a consequence, \eqref{eq:hc-hj-1} is equivalent to
\begin{equation}\label{eq:hc-hj-2}
	h^c_\nu(\lambda\otimes \nu,B)+\eps \geq 
	\frac{1}{\# Z(k)}
	E^\infty\big(u_k;R_\nu(k)\big)
\end{equation}
for all $k\in\N$. Analogously, we can also extend the elementary jump function $s_{\lambda,\nu}$ used in the definition of $h^j$
periodically to $s_{\lambda,\nu,k}\in SBV(R_\nu(k);\R{d})$, again without creating jumps at the interfaces since $s_{\lambda,\nu}$ is constant in directions perpendicular to $\nu$.

Now choose functions $\varphi_k\in C_c^\infty(R_\nu(k);[0,1])$ such that
\[
	\varphi_k=1~~\text{on $R_\nu(k-1)$}~~~\text{and}~~~|\nabla\varphi_k|\leq 2~~\text{on $R_\nu(k)\setminus R_\nu(k-1)$}
\]
Defining
\[
	\tilde{u}_k\coloneqq \varphi_k u_k+(1-\varphi_k) s_{\lambda,\nu,k},
\]
we obtain that $\tilde{u}_k=s_{\lambda,\nu,k}$ on $\partial R_\nu(k)$, 
$\tilde{u}_k=u_k$ on $R_\nu(k-1)$ and 
\[
\begin{aligned}
	|D\tilde{u}_k|(R_\nu(k)\setminus R_\nu(k-1))|
	&\leq 
	2 \norm{u_k-s_{\lambda,\nu,k}}_{L^1(R_\nu(k)\setminus R_\nu(k-1);\R{d})}
	+ |D\tilde{u}_k-Ds_{\lambda,\nu,k}|(R_\nu(k)\setminus R_\nu(k-1)) \\
	&\leq 2(N-1)(2k+1)^{N-2}
	\norm{u-s_{\lambda,\nu}}_{BV(Q_\nu;\R{d})}.
\end{aligned}
\]
Since $\# Z(k)=(2k+1)^{N-1}$, we conclude that 
$\frac{1}{\# Z(k)}|D\tilde{u}_k|(R_\nu(k)\setminus R_\nu(k-1))|=O(1/k)\to 0$ as $k\to\infty$.
Using the Lipschitz properties of $W$ \eqref{W2} and $\psi$ \eqref{psiLip},
we can can thus replace $u_k$ with $\tilde{u}_k$ in \eqref{eq:hc-hj-2}, with an error that converges to zero as $k\to \infty$:
\begin{equation}\label{eq:hc-hj-3}
	h^c_\nu(\lambda\otimes \nu,B)+\eps+O(1/k) \geq 
	\frac{1}{\# Z(k)} E^\infty\big(\tilde{u}_k;R_\nu(k)\big)
\end{equation}
Since $\tilde{u}_k=s_{\lambda,\nu}$ on $\partial R_\nu(k)$, we can define
\[
	\hat{u}_k(\tilde{x})\coloneqq \begin{cases}
		\tilde{u}_k\big((2k+1)\tilde{x}\big) ~~~& \text{if }\tilde{x}\in \frac{1}{2k+1}R_\nu(k),\\
		s_{\lambda,\nu}(\tilde{x})  ~~~& \text{if }(2k+1)\tilde{x}\in Q_\nu\setminus \frac{1}{2k+1}R_\nu(k),
	\end{cases}
\]	
without creating a jump at the interface between $\frac{1}{2k+1}R_\nu(k)$ and the rest.
As defined, $\hat{u}_k$ is now admissible for the infimum defining $h^j$, and
by a change of variables on the right-hand side of \eqref{eq:hc-hj-3}, we see that
\begin{equation}\label{eq:hc-hj-4}
\begin{aligned}
	h^c_\nu(\lambda\otimes \nu,B)+\eps+O(1/k) &
	\geq E^\infty\Big(\hat{u}_k;\frac{1}{2k+1}R_\nu(k)\Big)
	= E^\infty\big(\hat{u}_k;Q_\nu \big)
	\geq h^j(\lambda,B,\nu).
	\end{aligned}
\end{equation}
As $\eps>0$ and $k\in \N$ were arbitrary, \eqref{eq:hc-hj-4} implies that $h^c_\nu(\lambda\otimes \nu,B)\geq h^j(\lambda,B,\nu)$. The reverse inequality can be shown analogously.
\end{proof}

In the following proposition, we prove a sequential characterization of the relaxed energy densities defined in~\eqref{relaxed_energy_densities}.
To do so, we define the classes of sequences of competitors (see \eqref{competitors})
\begin{subequations}\label{seq_competitors}
\begin{eqnarray}
\cC^{\bulk}_{\seq}(A,B;Q) & \!\!\!\! \coloneqq & \!\!\!\! \big\{\{u_n\}\subset SBV(Q;\R{d}): \text{$u_n\wsto Ax$ in $BV$, $\nabla u_n\wsto B\cL^N$ in $\cM$}\big\}, \\
\!\!\!\! \!\!\!\! \!\!\!\! \cC^{\surface}_{\seq}(\lambda,\Lambda;Q_\nu) & \!\!\!\! \coloneqq & \!\!\!\! \big\{\{u_n\}\subset SBV(Q_\nu;\R{d}): \text{$u_n\wsto s_{\lambda,\nu}$ in $BV$, $\nabla u_n\wsto\Lambda\cL^N$ in $\cM$}\big\}.\quad\quad
\end{eqnarray}
\end{subequations}
\begin{proposition}\label{prop:seq-densities}
Suppose that \eqref{W1}--\eqref{W3} and \eqref{psi1}--\eqref{psi3} hold true. 
Then
\begin{subequations}\label{seq_char}
\begin{eqnarray}
H(A,B) &\!\!\!\! = &\!\!\!\! 
\inf\Big\{\liminf_{n\to\infty} E(u_n;Q): \{u_n\}\in \cC^{\bulk}_{\seq}(A,B;Q)\Big\}; \label{defH-seq}\\
h^j(\lambda,\Lambda,\nu)  &\!\!\!\! = &\!\!\!\! 
\inf\Big\{\liminf_{n\to\infty} E^{\infty}(u_n;Q_\nu): \{u_n\}\in \cC^{\surface}_{\seq}(\lambda,\Lambda;Q_\nu)\big\}; \label{201}\\
h^c(A,B) &\!\!\!\! = &\!\!\!\! 
\inf\Big\{\liminf_{n\to\infty} E^{\infty}(u_n;Q): \{u_n\}\in \cC^{\bulk}_{\seq}(A,B;Q)\Big\}. \label{202}
\end{eqnarray}
\end{subequations}
\end{proposition}
\begin{proof}
The formulae \eqref{defH-seq} and \eqref{202} are obtained in the same way as in \cite[Proposition~3.1]{ChoFo97} (for the latter, notice that $W=W^\infty$ is an admissible choice in \eqref{defH-seq}); formula \eqref{201} is obtained in the same way as in \cite[Proposition~4.1]{ChoFo97}.
\end{proof}

\subsection{Nonlinear transformation of measures and area-strict convergence}\label{sec32}

The following shorthand notation will prove useful below.
\begin{defin}[nonlinear transformation of measures]\label{def:dh}
For any Borel set $U\subset \R{N}$, any Borel function $h\colon \R{m}\to \R{}$ with strong recession function $h^\infty$ and any Radon measure
$\mu\in \cM(U;\R{m})$,
we define
\[
	\int_U \de h(\mu)\coloneqq \int_U h\Big(\frac{\de\mu}{\de \cL^N}\Big)\,\de |\mu|
	+\int_U h^\infty \Big(\frac{\de\mu^s}{\de|\mu^s|}\Big)\,\de |\mu^s|,
\]  
where $\mu^s$ denotes the singular part of the Radon-Nikodym decomposition of $\mu$
with respect to the Lebesgue measure $\cL^N$: $\mu=\frac{\de\mu}{\de\cL^N}\cL^N+\mu^s$.
\end{defin}
\begin{defin}[area-strict convergence, cf.~\cite{KriRi10a}]
For a Borel set $V$,
a sequence $(G_k)\subset \cM(V;\R{d\times N})$ and $G\in \cM(V;\R{d\times N})$, we say that $G_k\to G$ area-strictly if $G_k\wsto G$ in $\cM$ and
\[
	\int_V \de a(G_k)\to \int_V \de a(G),\quad\text{where}\quad a(\xi)\coloneqq \sqrt{1+\abs{\xi}^2},~\xi\in\R{d\times N}.
\]
Analogously, if $V$ is open, for a sequence $(g_k)\subset BV(V;\R{d})$ and $g\in BV(V;\R{d})$, we say that $g_k\to g$ area-strictly if $g_k\wsto g$ in $BV$ and $\int_V \de a(Dg_k)\to \int_V \de a(Dg)$.
\end{defin}

The following lemma is a generalized Reshetnyak continuity theorem; see \cite[Theorem~4]{KriRi10a} or \cite[Theorem~10.3]{Ri18B}.
\begin{lemma}\label{lem:astrict-cont} 
If $H\colon\R{d\times N}\times \R{d\times N}\to \R{}$ is continuous
and has a recession function in the strong uniform sense (see Proposition~\ref{prop:vvstrongrecession}),
then the functional defined on $\cM(\Omega;\R{d\times N})\times\cM(\Omega;\R{d\times N})$ by
\[
		(F,G)\mapsto \int_\Omega \de H(F,G) 
\]
is sequentially continuous with respect to the area-strict convergence of measures.
\end{lemma}
We also need the following well-known lemma 
combining area-strict approximations in $BV$ and $\cM$.
\begin{lemma}\label{prop:astrict-approx}
Let $(g,G)\in mSD$. 
Then there exists a sequence $\{(g_k,G_k)\}\subset W^{1,1}(\Omega;\R{d})\times L^1(\Omega;\R{d\times N})$
such that $g_k\to g$ area-strictly in $BV$ and $G_k\to G$ area-strictly in $\cM$.
\end{lemma}
\begin{proof}
The sequences $\{g_k\}$ and $\{G_k\}$ can be defined separately, essentially by mollification. 
As to $\{g_k\}$, see, for instance, 
\cite[Lemma~11.1]{Ri18B}, while the case of $\{G_k\}$ is simpler. 
\end{proof}

\section{Proof of Theorem~\ref{thm:representation}}
The proof of Theorem~\ref{thm:representation} is divided into two parts, each of which is carried out in the following section.
\subsection{Upper bound}\label{sec:upper}
\begin{proof}
We have to show that $I(g,G;\Omega)\leq J(g,G;\Omega)$, which is equivalent to the existence of
a ``recovery'' sequence $\{u_n\}$ admissible in the definition of $I$, i.e., such that $u_n\wsto (g,G)$ in $mSD$ and $E(u_n;\Omega)\to J(g,G;\Omega)$.
The proof here is presented using a series of auxiliary results collected below.

First observe that based on Proposition~\ref{prop:H-hj-hc}, 
our candidate $J$ for the limit functional, introduced in \eqref{defJ} using $H$, $h^j$, and $h^c$, can
be expressed as a standard integral functional of the measure variable $(Dg,G)$ using only $H$ and its recession function $H^\infty$ as integrands (Proposition~\ref{prop:J-with-H-only}). 
As $H$ is continuous and its recession function exists in a strong enough sense (cf.~Proposition~\ref{prop:vvstrongrecession}), $J$ is
sequentially continuous with respect to the area-strict convergence of measures (Lemma~\ref{lem:astrict-cont}). 
Since any $(g,G)\in mSD$ 
can be approximated area-strictly by sequences in $W^{1,1}\times L^1$
(Lemma~\ref{prop:astrict-approx}),
a diagonalization argument allows us to reduce the construction of the recovery sequence to the case
$(g,G)\in W^{1,1}\times L^1\subset SBV\times L^1$. This special case was already obtained in \cite{ChoFo97}, see Remark~\ref{rem:ChoFo}.
\end{proof}

\begin{proposition}\label{prop:J-with-H-only}
Suppose that \eqref{W1}--\eqref{W3} and \eqref{psi1}--\eqref{psi3} hold.
Then with the notation of Definition~\ref{def:dh},
\[
	J(g,G;\Omega)=\int_\Omega \de H(Dg,G)
\]
where $J$ is defined in \eqref{defJ} and $H$ is given by \eqref{def-H}.
\end{proposition}
\begin{proof}
Clearly, $H\Big(\frac{\de(Dg,G)}{\de\cL^N}\Big)=H\Big(\nabla g,\frac{\de G}{\de \cL^N}\Big)$.
In addition, 
\[
	|(Dg,G)^s|=\Big(1+\frac{\de G}{\de |D^s g|}\Big)|D^s g|+ 
	\theta=\Big(1+\frac{\de G}{\de |D^c g|}\Big)|D^c g|+\Big(1+\frac{\de G}{\de |D^j g|}\Big)|D^j g|+ 
	|G^s_g|,
\]
since $|G^s_g|$ and $|Dg^s|+\cL^N$ are mutually orthogonal by definition of $G^s_g$,
and the Cantor and jump parts of $Dg$ are mutually orthogonal as well.
Since $H^\infty$ is positively $1$-homogeneous, the definition of $J$ implies the 
asserted representation once we use Proposition~\ref{prop:H-hj-hc} to replace $H^\infty$ by $h^j$ and $h^c$, respectively.
\end{proof}

We need the following regularity properties of $H$, in particular at infinity.
\begin{proposition}\label{prop:vvstrongrecession}  
Suppose that \eqref{W1}, \eqref{W3} and \eqref{psi1}--\eqref{psi3} hold.
Then $H$ is globally Lipschitz and the recession function $H^\infty$ exists in the strong uniform sense, i.e., the limit
\begin{equation}\label{str-uni-rec}
	\lim_{\substack{(A',B')\to (A,B)\\ t\to+\infty}} \frac{H(tA',tB')}{t} 
\end{equation}
exists for all $(A,B)\in \R{d\times N}\times \R{d\times N} \setminus \{(0,0)\}$.
\end{proposition}
\begin{proof} 
The Lipschitz property of $H$ was proved in \cite[Theorem~2.10]{BaMaMoOwZa22a} (the case $p=1$).
Concerning \eqref{str-uni-rec}, first observe that since $H$ is Lipschitz with some constant $L>0$,
\begin{equation}\label{p:sur-00}
	\abs{\frac{H(tA',tB')}{t}-H^\infty(A,B)}
	\leq L \abs{(A',B')-(A,B)} 
	+\abs{\frac{H(tA,tB)}{t}-H^\infty(A,B)}
\end{equation}
Here, $H^\infty(A,B)=\limsup_{t\to\infty} \frac{1}{t}H(tA,tB)$ as before.
It, therefore suffices to show that 
\[
	\lim_{t\to+\infty} \frac{H(tA,tB)}{t}= H^\infty(A,B).
\]
We claim that in fact, we even have that
\begin{equation}\label{p:sur-01}
	\abs{\frac{H(tA,tB)}{t}-H^\infty(A,B)}\leq 
	C(A,B) \Big( \frac{1}{t^\alpha}+ \frac{1}{t}\Big)
	\quad\text{for all}~~t>0,~A,B\in \R{d\times N},
\end{equation}
where $C(A,B)>0$ is a constant independent of $t$ 
and $H^\infty(A,B)=\limsup_{t\to+\infty} \frac{1}{t}H(tA,tB)$.

For a proof of \ref{p:sur-01}, first fix $\eps>0$ and choose
$\eps$-almost optimal sequence $\{u_{t,n}\}_n$ for the sequential characterization of $H(tA,tB)$ in Proposition~\ref{prop:seq-densities}, dependent on $t>0$ (and $A$,$B$). 
This choice yields that
\begin{equation}\label{p:sur-03}
	H(tA,tB) +\eps \geq \int_{Q} W(\nabla u_{t,n})\,\de x+\int_{Q\cap S_{u_{t,n}}} \psi([u_{t,n}],\nu_{u_{t,n}})\,\de\cH^{N-1}(x).
\end{equation}
The sequence $v_{t,n}\coloneqq t^{-1}u_{t,n}$ then is also in the class of admissible sequences for the sequential characterization of $h^c(A,B)$ in Proposition~\ref{prop:seq-densities}, and since $h^c(A,B)=H^\infty(A,B)$ by Proposition~\ref{prop:H-hj-hc},
this entails that
\begin{equation}\label{p:sur-04}
	H^\infty(A,B) 
	\leq \int_{Q} W^\infty(\nabla v_{t,n})\,\de x
		+\int_{Q\cap S_{v_{t,n}}} \frac{1}{t}\psi([v_{t,n}],\nu_{v_{t,n}})\,\de\cH^{N-1}(x),
\end{equation}
where we exploited that $W^\infty$ and $\psi(\cdot,\nu)$ are positively $1$-homogeneous.
Multiplying \eqref{p:sur-03} by~$t^{-1}$ and combining it with \eqref{p:sur-04} yields
\begin{equation}\label{p:sur-05}
	H^\infty(A,B)-\frac{1}{t}H(tA,tB) \leq \frac{1}{t}\eps+ \int_{Q} 
	\left(W^\infty(\nabla v_{t,n})- \frac{1}{t}W(t\nabla v_{t,n}))\right)\de x.
\end{equation}
Analogously, we can also choose and $\eps$-almost optimal sequence $\tilde{v}_h$ for 
the sequential characterization of $h^c(A,B)=H^\infty(A,B)$, which makes $\tilde{u}_{t,h}\coloneqq t\tilde{v}_h$ admissible for the sequential characterization of $H(tA,Bt)$. With this, get that
\begin{equation}\label{p:sur-06}
	\frac{1}{t}H(tA,tB)-H^\infty(A,B) \leq \eps+ \int_{Q} 
	\left(\frac{1}{t} W(t\nabla \tilde{v}_{t,n}))-W^\infty(\nabla \tilde{v}_{t,n})\right) \de x.
\end{equation}
The right hands sides of \eqref{p:sur-05} and \eqref{p:sur-06} can now be estimated in the same fashion: 
by \eqref{W1} and the homogeneity of $W^\infty$ we have that
\[
	\abs{\frac{W(tA)}{t}-W^\infty(A)}\leq \abs{\frac{W(tA)}{t}}+\abs{W^\infty(A)} \leq
	C_W\bigg(\!\abs{A}+\frac{1}{t}\bigg)+C_W\abs{A}\leq 3C_W \frac{1}{t}\quad\text{if $t\abs{A}<1$}.
\]
This is exactly the case excluded in \eqref{W3}, so that together with \eqref{W3}, we obtain that
\begin{equation}\label{p:sur-08}
\begin{aligned}
	\abs{\frac{W(tA)}{t}-W^\infty(A)}
	& \leq \frac{c\abs{A}^{1-\alpha}}{t^\alpha}+ 3C_W \frac{1}{t} \\
  &\leq \frac{c(1+\abs{A})}{t^\alpha}+ 3C_W \frac{1}{t}
	\quad\text{for all $t>0$ and $A\in \R{d\times N}$},
	\end{aligned}
\end{equation}
since $0<\alpha<1$.
Moreover,
\eqref{p:sur-03} implies that $\norm{\nabla v_{t,n}}_{L^1}=t^{-1}\norm{\nabla u_{t,n}}_{L^1}$ is equi-bounded for $t\geq 1$ since $H$ is globally Lipschitz, $\psi\geq 0$ and $W$ is coercive by \eqref{W1}.
Similarly, $\norm{\nabla \tilde{v}_{t,n}}_{L^1}$ is equi-bounded.
Thus, 
\begin{equation}\label{p:sur-09}
	M(A,B)\coloneqq \sup_{t\geq 1}\sup_{n\in \N} \big(
	\norm{\nabla v_{t,n}}_{L^1}+\norm{\nabla \tilde{v}_{t,n}}_{L^1}\big)<\infty,
\end{equation}
Now we can use \eqref{p:sur-08} to obtain upper bounds for the right-hand sides 
of \eqref{p:sur-05} and \eqref{p:sur-06} and combine them. 
By \eqref{p:sur-09}, this yields that
\begin{equation}\label{p:sur-10}
	\abs{\frac{1}{t}H(tA,tB)-H^\infty(A,B)} \leq \eps\Big(\frac{1}{t}+1\Big)+
	\frac{c(1+M(A,B))}{t^\alpha}+ 3\cL^N(\Omega)C_W \frac{1}{t}
\end{equation}
for all $t\geq 1$. Since $\eps>0$ was arbitrary, \eqref{p:sur-10} implies \eqref{p:sur-01}.
\end{proof}

\subsection{Lower Bound}\label{sec:lower bound}
Our proof of the lower bound relies on the following lower semicontinuity property.
\begin{proposition}\label{prop:H-ws-lsc}
Assume that \eqref{W1}--\eqref{W3} and \eqref{psi1}--\eqref{psi3} hold.
Then the integrand $H$ defined in \eqref{def-H} is quasiconvex-convex in the sense that for all $A,B\in \R{d\times N}$,
\[
\begin{aligned}
	&\ave_Q H(A+\nabla v,B+w)\,\de x\geq H(A,B) \\
	&\;\text{for all $(v,w)\in W_0^\infty(Q;\R{d})\times L^\infty(Q;\R{d\times N})$ with $\ave_Q w\,\de x=0$.}
\end{aligned}
\]
Moreover, the functional $(g,G)\mapsto \int_\Omega \de H(Dg,G)$ is sequentially lower semi-continuous with respect to the convergence in \eqref{SBV_conv}. 
\end{proposition}
\begin{proof}
We will first show that $(g,G)\mapsto \int_\Omega H(\nabla g,G)\,\de x$
is sequentially lower semi-continuous with respect to weak convergence in $W^{1,1}\times L^1$.
Take $(g,G)\in W^{1,1}(\Omega;\R{d})\times L^1(\Omega;\R{d\times N})$ and $\{(g_n,G_n)\}\subset W^{1,1}(\Omega;\R{d})\times L^1(\Omega;\R{d\times N})$ with 
$(g_n,G_n)\rightharpoonup (g,G)$ weakly in $W^{1,1}\times L^1$.
By \cite[Theorem 2.16]{ChoFo97} (recovery sequence for the case of $I_1$ therein), 
for each $n$ there exists a sequence $\{u_{n,k}\}_k\subset SBV(\Omega;\R{d})$ such that as $k\to\infty$,
\[
	\text{$u_{n,k}\to g_n$  in $L^1(\Omega;\R{d})$\quad and \quad$\nabla u_{n,k}\wsto G_n$ in $\cM(\Omega;\R{d\times N})$,} 
\]
and
\[
	E(u_{n,k};\Omega)\to \int_\Omega H(\nabla g_n,G_n)\,\de x.
\]
In addition, we may assume that up to a (not relabeled) subsequence
$$\liminf_{n\to\infty} \int_\Omega H(\nabla g_n,G_n)\,\de x=\lim_{n\to\infty} \int_\Omega H(\nabla g_n,G_n)\,\de x<+\infty.$$
Since $E$ is coercive as a consequence of \eqref{W1} and \eqref{psi1},
the latter implies that $\{Du_{n,k}\}$ is equi-bounded in $\cM(\Omega;\R{d\times N})$. 
We can therefore find a diagonal subsequence $u_{n,k(n)}$ with $k(n)\to \infty$ fast enough, such that
\[
	(u_{n,k(n)},\nabla u_{n,k(n)})\wsto (g,G)\text{ in $mSD$ and }
	\lim_{n\to\infty} E(u_{n,k(n)};\Omega)=\lim_{n\to\infty} \int_\Omega H(\nabla g_n,G_n)\,\de x.
\]
Since the sequence $\{u_{n,k(n)}\}_n$ is admissible for the lower bound in
\cite[Theorem~2.16]{ChoFo97} (for the functional $I_1$),
we conclude that
\[
	\lim_{n\to\infty} \int_\Omega H(\nabla g_n,G_n)\,\de x
	=\lim_{n\to\infty} E(u_{n,k(n)};\Omega)\geq \int_\Omega H(\nabla g,G)\,\de x,
\]
i.e., the functional with integrand $H$ is weakly lower semicontinuous in $W^{1,1}\times L^1$.
Since $H$ also has at most linear growth 
by Proposition~\ref{prop:vvstrongrecession} and is non-negative as a consequence of \eqref{W1},
\cite[Theorem~1.1]{CaRiZa10a} (see also \cite{CaRiZa11}) implies that $H$ is quasiconvex-convex.

It remains to show that 
$(g,G)\mapsto \int_\Omega \de H(Dg,G)$ is sequentially lower semi-continuous with respect to the convergence \eqref{SBV_conv}. 
This follows from
\cite[Theorem~1.7]{ArDPRi20a}. Here, notice that with $\mathcal{A}\coloneqq (\operatorname{Curl},0)$,
 $\mathcal{A}(\nabla g_k,G_k)^\top=0$ in the sense of distributions, and 
the $\cA$-quasiconvexity of $H$ for this special case is equivalent to quasiconvexity-convexity of $H$.
The latter can equivalently be tested with periodic functions on the simply connected~$U$ where all curl-free fields are gradients.
\end{proof}

\begin{proof}[Proof of Theorem~\ref{thm:representation}, lower bound]
	Let $H$ be the integrand in \eqref{def-H}. Moreover, let $(g,G) \in mSD$ and $\{v_n\} \subset SBV(\Omega;\mathbb R^d)$ be such that
	$v_n \wsto 
	(g,G)$ in the sense of \eqref{SBV_conv}. 
Observing	that for each $n$, $v_n$ can be interpreted as a constant sequence converging to itself in $mSD$, by Proposition~\ref{prop:H-hj-hc} and \cite[Theorem 2.16]{ChoFo97} (its lower bound for the case of $I_1$ therein), we have that
\begin{equation}\label{CF-lsc} 
	\int_{\Omega} \de H(Dv_n, \nabla v_n\cL^N) \leq 
	\int_{\Omega}W(\nabla v_n) 	\de x +\int_{\Omega\cap S_{v_n}} \psi([v_n], \nu_{v_n}) \de {\mathcal H}^{N-1}(x).\end{equation}
In addition, $(u,G)\mapsto \int_\Omega \de H(Du,G)$ is weak$^*$-sequentially lower semicontinuous in $mSD$ 
by Proposition~\ref{prop:H-ws-lsc}. In particular,
\begin{equation}\label{dH-lsc} 
		\int_{\Omega} \de H(Dg, G) \leq \liminf_{n\to \infty} \int_{\Omega} \de H(Dv_n, \nabla v_n\cL^N).
\end{equation}
Taking Proposition \ref{prop:J-with-H-only} into account, the lower bound inequality now follows from \eqref{dH-lsc} and \eqref{CF-lsc}:
	\begin{align*}
		\label{chineqlb}
		\begin{aligned}
		J(g,G;\Omega)&=\int_{\Omega} \de H(Dg, G)
	\leq	\liminf_{n\to \infty}\int_{\Omega} \de H(Dv_n, \nabla v_n\cL^N)  \\
	&	\leq	\liminf_{n\to \infty}\int_{\Omega}W(\nabla v_n) \de x +\int_\Omega \psi([v_n], \nu_{v_n}) \de {\mathcal H}^{N-1}(x)= \liminf_{n \to\infty}E(v_n). \qedhere
			\end{aligned}
			\end{align*} 
\end{proof}

\section{Relaxation under trace constraints}\label{sec:withBC}

Let $\Omega'$ be a bounded Lipschitz domain such that $\Omega \subset \Omega'$, and let 
\[
	\Gamma:= \Omega' \cap \partial \Omega.
\]
Let $u_0 \in W^{1,1}(\Omega';\mathbb R^d)$ and let $(g,G) \in mSD$.
The relaxed functional subject to the Dirichlet condition $u=u_0$ on $\Gamma$
is defined as
\begin{equation}\label{IcOmega}I_\Gamma(g,G;\Omega)\coloneqq 
\inf \left\{\liminf_{n\to \infty} E(u_n;\Omega) ~\left|~  
\begin{aligned} & u_n \in SBV(\Omega;\mathbb R^d), u_n=u_0 \hbox{ on }\Gamma, \\
& u_n \wsto g \hbox{ in } BV(\Omega;\mathbb R^d),\\
&\nabla u_n \wsto G \hbox{ in } \mathcal M(\Omega\cup \Gamma;\mathbb R^{d \times N})
\end{aligned}\right.\right\},
\end{equation}
where, for every open subset $A$ of $\Omega' $, $E(\cdot; A)$ is the functional given by \eqref{001}, with $W$ and $\psi$ satisfying \eqref{W1}-\eqref{W3} and \eqref{psi1}-\eqref{psi3},

We have the following integral representation for $I_\Gamma$.
\begin{theorem}\label{lbclamped}
Let $\Omega\subset\R{N}$ be a bounded Lipschitz domain and assume that \eqref{W1}--\eqref{W3} and  
\eqref{psi1}--\eqref{psi3} hold. Moreover, let $\Omega'\supset \Omega$ be a bounded domain and 
$u_0 \in W^{1,1}(\Omega';\mathbb R^d)$. In addition, for $\Gamma\coloneqq \Omega' \cap \partial \Omega$ assume that $\cH^{N-1}(\overline{\Gamma}\setminus \Gamma)=0$. 
Then,
\[
	I_\Gamma(g,G;\Omega)=J_\Gamma(g,G;\Omega)\quad
	\text{for every $(g,G) \in mSD$,} 
\]
where
\[
\begin{split}
J_\Gamma(g,G;\Omega)\coloneqq & \int_\Omega \de H(g,G) \\
&\,+ \int_{\Gamma}  H^\infty\left(\frac{\de ([g-u_0]\otimes \nu_{\Gamma}\,\mathcal H^{N-1}\res \Gamma, G)}{\de |([g-u_0]\otimes \nu_{\Gamma}\,\mathcal H^{N-1}\res \Gamma, G)|}
\right) \de | ([g-u_0]\otimes \nu_{\Gamma}\,\mathcal H^{N-1}\res\Gamma, G)|, 
\end{split}
\]
and $H$ is the function defined in \eqref{def-H}.
\end{theorem}
The proof will be given in two parts. We immediately start with the lower bound,
and the proof of the upper bound will follow after an auxiliary result needed there.

\begin{proof}[Proof of Theorem~\ref{lbclamped}, the lower bound]

We have to show that $I_\Gamma(g,G;\Omega)\geq J_\Gamma(g,G;\Omega)$. 
For every $k \in \mathbb N$, let $\Omega_k:=\{x \in \mathbb R^N: {\rm dist}(x, \overline\Omega) \leq \tfrac{1}{k}\}$, and consider
\[
\Omega'_k:= \Omega_k \cap \Omega'.\]
Thus $\Gamma= \Omega'_k\cap \partial \Omega$,  for every $k$ and $\Omega'_k$ shrinks to $\Omega \cup \Gamma$ as $k \to\infty$.
As for $I_\Gamma(g,G;\Omega)$, define for every $k \in \mathbb N$
\begin{align*}\hat I_\Gamma(g,G;\Omega'_k):=& 
 \int_{\Omega'_k\setminus \Omega} W(\nabla u_0)\, \de x \\
	&+\inf \Big\{\liminf_{n\to \infty} E(u_n;\Omega): 
	\begin{aligned}[t]
	&u_n \in SBV(\Omega'_k;\mathbb R^d), u_n=u_0 \hbox{ on }\partial \Omega, \\
	&u_n \wsto g \hbox{ in } BV(\Omega;\mathbb R^d), \nabla u_n \wsto G \hbox{ in } \mathcal M(\Omega\cup 	\Gamma;\mathbb R^{d \times N})\Big\}.
	\end{aligned}
\end{align*}
Thus
\begin{align}\label{IeqIhat} I_\Gamma(g,G; \Omega) = \hat I_\Gamma(g, G; \Omega'_k)- \int_{\Omega'_k\setminus \Omega} W(\nabla u_0) \de x.
\end{align}

On the other hand,
\begin{align*}
\hat I_\Gamma(g,G; \Omega'_k)= \inf\Big \{\liminf_{n\to\infty} E(v_n;\Omega'_k): 
v_n \in SBV(\Omega'_k;\mathbb R^d), v_n = u_0 \hbox{ in } \Omega'_k \setminus \overline\Omega,  \\
v_n \wsto \hat g \hbox{ in }BV(\Omega'_k;\mathbb R^d), \nabla v_n \wsto \hat G \hbox{ in }\mathcal M(\Omega'_k;\mathbb R^{d \times N})
\Big\},
\end{align*}
where
\[
	\hat g\coloneqq \begin{cases}
	 g & \text{in $\Omega$,}\\
 	u_0 & \text{in $\Omega'_k \setminus \overline\Omega$}
	\end{cases}
	\qquad\text{and}\quad
	\hat G\coloneqq \begin{cases}
	G & \text{in $\Omega \cup \Gamma$,}\\
	\nabla u_0 & \text{in $\Omega'_k \setminus \overline\Omega$.} 
	\end{cases}
\]
In particular, 
\[
	D\hat g\lfloor_\Gamma = [g-u_0]\otimes \nu_{\Gamma}\,\mathcal H^{N-1}\res\Gamma.
\]
Clearly, for every $\Omega'_k$,
\begin{align}\label{ineqmain}
\begin{aligned}
&\hat I_\Gamma(g,G; \Omega'_k) \\
\geq& I(\hat g,\hat G;\Omega'_k)= \int_{\Omega'_k} \de H(\hat g, \hat G) \\ 
\geq& \int_\Omega \de H(g,G)+ \! \int_{\Gamma} \! H^\infty \! \left(\frac{\de ([g-u_0]\otimes \nu_{\Gamma}\, \mathcal H^{N-1}\res\Gamma, G)}{\de |([g-u_0]\otimes \nu_{\Gamma}\,\mathcal H^{N-1}\res\Gamma, G)|}\right) \de | ([g-u_0]\otimes \nu_{\Gamma}\,\mathcal H^{N-1}\res\Gamma, G)|,  
\end{aligned}
\end{align}
where $I(\hat g,\hat G;\Omega'_k)$ is the functional introduced in \eqref{defI}, and in the equality we have exploited Theorem \ref{thm:representation} and Remark \ref{rem:J_as_int_dH}.
The proof is concluded by letting $k\to \infty$, in the above inequality, taking into account \eqref{IeqIhat} and the fact that 
$\lim_{k \to \infty} \int_{\Omega'_k\setminus \Omega} W(\nabla u_0) \, \de x
=0$. 
\end{proof}

Below, we will reduce the construction of the recovery sequence needed for the upper bound to that of Theorem~\ref{thm:representation}. This relies on the following lemma.
\begin{lemma}[domain shrinking {\cite[Lemma 3.1]{KroeVa23a}}]\label{lem:shrink}
Let $\Omega\subset\R{N}$ be a bounded Lipschitz domain.
Then there exists an open neighborhood $U\supset \overline{\Omega}$ and 
a sequence of maps $\{\Psi_j\}\subset C^\infty(\overline{U};\R{N})$ such that
for every $j\in\N$,
\begin{equation}
	\Psi_j\colon\overline{U}\to \Psi_j(\overline{U})~~~\text{is invertible and}~~~ 
	\Psi_j(\Omega)\subset\subset \Omega.
\end{equation}
In addition, $\Psi_j\to \operatorname{id}$ in $C^m(\overline{U};\R{N})$ as $j\to\infty$, for all $m\in \N\cup \{0\}$. 
\end{lemma}
\begin{proof}
This is the case $\Gamma=\emptyset$ in \cite{KroeVa23a}.
The statement there has $\Psi_j$ only defined on $\overline\Omega$, but the proof also provides 
the extension to $\overline{U}$ (as long as $\overline{U}$ is still fully 
covered by the union of $\Omega$ and the open cuboids covering $\partial\Omega$ in which $\partial\Omega$ can be seen as a Lipschitz graph).
\end{proof}
\begin{remark}
If $\Omega$ is strictly star-shaped with respect to some $x_0\in\Omega$, 
Lemma~\ref{lem:shrink} is easy to show with
$\Psi_j(x)\coloneqq x_0+\frac{j}{j+1}(x-x_0)$.
The proof of \cite[Lemma 3.1]{KroeVa23a} for the general case glues local constructions near the boundary using a decomposition of unity, exploiting that everything happens uniformly $C^1$-close to the identity to preserve invertibility.
\end{remark}

\begin{proof}[Proof of Theorem~\ref{lbclamped}, the upper bound]
We have to show that $I_\Gamma(g,G;\Omega)\leq J_\Gamma(g,G;\Omega)$,
for each $(g,G)\in BV(\Omega;\R{d})\times \cM(\Omega\cup\Gamma;\R{d\times N})$. 
For this, it suffices to find a recovery sequence, i.e., a sequence
$(u_n)$ admissible in the definition of $I_\Gamma(g,G;\Omega)$ such that
$E(u_n;\Omega)\to J_\Gamma(g,G;\Omega)$. In particular, we must have $u_n=u_0$ on $\Gamma$ in the sense of traces in $BV$. The proof is divided into three steps. 
In the first two steps, we define a suitable approximating sequence of limit states $(\hat{g}_j,\hat{G}_j)$ such that
$(\hat{g}_j,\hat{G}_j)\wsto (g,G)$ in $mSD$, $J_\Gamma(\hat{g}_j,\hat{G}_j;\Omega)\to J_\Gamma(g,G;\Omega)$
and $\hat{g}_j=u_0$ on $\Gamma$.
In the final step, we will then use the upper bound in Theorem~\ref{thm:representation}, which for each $j$ gives a ``free'' recovery sequence $\{u_{j,n}\}_n\subset BV$ for $I(\hat{g}_j,\hat{G}_j;\Omega)$ that again can be modified to match the trace of its weak$^*$ limit $\hat{g}_j$ on $\Gamma$.
The assertion then follows by a diagonal subsequence argument.

{\bf Step 1: Approximating limit states $(g_j,G_j)$ with values ``close'' to $u_0$ near $\Gamma$.}
Choose a bounded neighborhood $U$ of $\Omega$ according to Lemma~\ref{lem:shrink}
and an extension 
\[
\begin{aligned}
	&\tilde{g}\in BV(U;\R{d}) ~~~\text{with}~~~\tilde{g}|_\Omega=g,~~~|D\tilde{g}|(\partial\Omega)=0,
	~~~\tilde{g}\in W^{1,1}(U\setminus \overline\Omega;\R{d}).
\end{aligned}
\]
With this, we define
\[
\begin{aligned}
	BV(U;\R{d})\ni g_0\coloneqq& \chi_{\Omega}g+\chi_{U\cap (\Omega'\setminus \overline{\Omega})} u_0+
	\chi_{U\setminus \overline{\Omega}'} \tilde{g}, \\
	\cM(U;\R{d\times N})\ni G_0\coloneqq& \chi_{\Omega\cup \Gamma} G.
\end{aligned}
\]
In particular, with the outer normal $\nu_\Gamma$ to $\partial\Omega$ on $\Gamma$,
\[
	G_0\res(\Omega\cup \Gamma)=G,\qquad Dg_0\res(\Omega\cup \Gamma)
	=Dg\res\Omega+(u_0-g)\otimes \nu_\Gamma\, \cH^{N-1}\res\Gamma,
\]
$G_0\res(U \setminus (\Omega\cup \Gamma))=0$, $g_0|_\Omega=g$,
and
$g_0$ jumps at $\Gamma$ from (the trace of) $g$ to $u_0$
and at $U\cap(\partial \Omega'\setminus \Omega)$ from 
$\tilde{g}$ to $u_0$.

With the maps $\Psi_j$ from Lemma~\ref{lem:shrink}, we define
\[
	\Phi_j\coloneqq \Psi_j^{-1}~~~\text{and}~~~(g_j,G_j)\coloneqq \big(g_0\circ \Phi_j,(G_0\circ \Phi_j)\nabla \Phi_j\big)
	\in BV(\Psi_j(U);\R{d})\times \cM(\Psi_j(U);\R{d\times N}).
\]
Here, in the definition of $G_j$, $G_0\circ \Phi_j$ is the measure defined as
$(G_0\circ \Phi_j)(A)\coloneqq G_0(\Phi_j(A))$ for all Borel sets $A\subset \Psi_j(U)$,
and $\nabla\Phi_j\in C^0(\overline{U};\R{N\times N})$ is interpreted as a continuous density function attached to it by matrix multiplication from the right. Altogether, 
$G_j$ is the measure satisfying $\de G_j(z)=\de (G_0\circ \Phi_j)(z)\nabla \Phi_j(z)$,
similar to $Dg_j$ which satisfies
$\de Dg_j(z)=(Dg\circ \Phi_j)(z) \nabla\Phi_j(z)$ by the chain rule.
Also notice that as a consequence of Lemma~\ref{lem:shrink}  (where we only need the case $m=1$), for all $j$ big enough, $\nabla\Psi_j(x)$ is an invertible matrix for all $x\in U$,
$\overline{\Omega}\subset \Psi_j(U)$ and $\Phi_j(\partial\Omega)\cap \overline\Omega=\emptyset$.
Passing to a subsequence (not relabeled), we thus may assume that
\begin{equation}\label{ubBC-05}
	\Psi_j:U\to \Psi_j(U)\text{ is a diffeomorphism},~~~ 
	\overline{\Omega}\subset \Psi_j(U) ~~~\text{and}~~~
	\Phi_j(\partial\Omega)\cap \overline\Omega=\emptyset\quad \text{for all $j\in \N$}.
\end{equation}
We claim that the sequence $\{(g_j,G_j)\}_j$ has the following properties:
\begin{equation}\label{ubBC-09}
	\norm{T_\Omega(g_j-u_0)}_{L^1(\Gamma;\R{d})} 
	\underset{j\to\infty}{\longrightarrow} 0,\quad
	|G_j|(\Gamma)=|G_j|(\partial\Omega)=0,
\end{equation}
where $T_\Omega:BV(\Omega;\R{d})\to L^1(\partial\Omega;\R{d})$ denotes the trace operator,
\begin{equation}\label{ubBC-10}
	g_j|_{\Omega} \wsto g
	~~~\text{in}~BV(\Omega;\R{d}),\quad
	(Dg_j\res{\Omega},G_j\res{\Omega}) \,\wsto\, (Dg_0\res(\Omega\cup \Gamma),G_0)
	~~~\text{in $\cM(\overline\Omega;\R{d\times N})^2$}
\end{equation}
and
\begin{equation}\label{ubBC-11}
	\int_\Omega \de H(Dg_j,G_j)\underset{j\to\infty}{\longrightarrow}
	\int_{\Omega\cup \Gamma} \de H(Dg_0\res(\Omega\cup \Gamma),G).
\end{equation}

The second part of \eqref{ubBC-09} follows from the definition of $G_j$ because  
$\Phi_j(\partial\Omega)\subset U\setminus \overline{\Omega}$ and $|G_0|(U\setminus \overline{\Omega})=0$.
As to the first part of \eqref{ubBC-09},  
first notice that since $u_0\in W^{1,1}(U;\R{d})$, we do not have to 
distinguish between the inner and outer traces $T_\Omega u_0$ and 
$T_{U\setminus\Omega} u_0$ of $u_0$ on $\partial\Omega$.
Moreover,
\[
	g_j-u_0=\big(g_0\circ \Phi_j-u_0\circ \Phi_j\big)+\big(u_0\circ \Phi_j-u_0\big)
\]
and $u_0\circ \Phi_j\to u_0$ in $W^{1,1}(\Omega;\R{d})$,
so that
the asserted convergence of traces follows from the continuity of the trace operator
in $W^{1,1}$ once we see that 
$(g_0-u_0)\circ \Phi_j=0$ in some neighborhood of $\Gamma$ (which may depend on $j$).
The latter is trivial by definition of $g_0$ if
\begin{equation}\label{ubBC-35}
\begin{aligned}
	\Phi_j(\overline{\Gamma})\subset \Omega'\setminus \overline{\Omega}\quad\text{for all $j\in\N$};
\end{aligned}
\end{equation}
here, we already have that $\Phi_j(\overline{\Gamma})\cap \overline{\Omega}=\emptyset$.
We can therefore assume \eqref{ubBC-35} without loss of generality: otherwise, if
$\Phi_j(\overline{\Gamma})\not\subset \Omega'$, we can 
define $r(j)\coloneqq \frac{1}{2}\operatorname{Dist}(\Phi_j(\overline{\Gamma}),\Omega)>0$ and
take 
\[
\begin{aligned}
	\tilde{\Omega}'\coloneqq \Omega'\cup \big\{x\in \R{N}\mid 
	\dist{x}{\Phi_j(\overline{\Gamma})}<r(j)~~\text{for a $j\in\N$}\big\} & 
\end{aligned}
\]
instead of $\Omega'$.
Here, recall that $\Omega'$ is just an auxiliary object to define $\Gamma$ (and $g_0$ above, outside of $\overline\Omega$),
and by construction,
$\tilde{\Omega}'$ still has all the properties we required for $\Omega'$: $\tilde{\Omega}'\supset\Omega$ is a bounded domain and
$\tilde{\Omega}'\cap \partial\Omega=
\Gamma=\Omega'\cap \partial\Omega$.

For the proof of \eqref{ubBC-10} and \eqref{ubBC-11},
fix $\varphi\in C(\overline\Omega)$, continuously extended to $\varphi\in C(\R{N})$. By the definition of $(g_j,G_j)$ and the change of  
variables $x=\Phi_j(z)$,
we get that for every Borel set $V\subset U$,
\begin{equation}\label{ubBC-13}
\begin{aligned}
	&\int_\Omega \varphi(z) \de H(Dg_j,G_j)(z) \\
	=&\int_\Omega \varphi(z) \de H\big((Dg_0\circ \Phi_j)\nabla \Phi_j,(G_0\circ \Phi_j)\nabla \Phi_j\big)(z) \\
	=&\begin{aligned}[t]
		&\int_{\Phi_j(\Omega)} \varphi(\Psi_j(x)) H\Big(\nabla g_0 (\nabla \Psi_j)^{-1},\frac{\de G_0}{\de \cL^N}(\nabla \Psi_j)^{-1}\Big)\det(\nabla \Psi_j(x))\,\de x \\
		&+\int_{\Phi_j(\Omega)} \varphi(\Psi_j(x)) \de H^\infty\big(D^sg_0 (\nabla \Psi_j)^{-1},G^s_0 (\nabla \Psi_j)^{-1}\big)(x) \\
	\end{aligned} \\
	=& \begin{aligned}[t] 
		&\bigg(
		 \int_{\Phi_j(\Omega)\cap V} (\varphi\circ \Psi_j) H\Big(\nabla g_0 (\nabla \Psi_j)^{-1},\frac{\de G_0}{\de \cL^N}(\nabla \Psi_j)^{-1}\Big)\det(\nabla \Psi_j(x))\,\de x \\
		& \quad +\int_{\Phi_j(\Omega)\cap V} (\varphi\circ \Psi_j) \,\de H^\infty\big(D^sg_0 (\nabla \Psi_j)^{-1},G^s_0 (\nabla \Psi_j)^{-1}\big)(x)\bigg) \\
		&+\bigg(\int_{\Phi_j(\Omega)\setminus V} (\varphi\circ \Psi_j) H\Big(\nabla g_0 (\nabla \Psi_j)^{-1},\frac{\de G_0}{\de \cL^N}(\nabla \Psi_j)^{-1}\Big)\det(\nabla \Psi_j(x))\,\de x \\
			&\quad +\int_{\Phi_j(\Omega)\setminus V}  (\varphi\circ \Psi_j) \,\de H^\infty\big(D^sg_0 (\nabla \Psi_j)^{-1},G^s_0 (\nabla \Psi_j)^{-1}\big)(x) \bigg)
	\end{aligned}\\
	\eqqcolon & S_j(\varphi;V)+T_j(\varphi;V)	
\end{aligned}
\end{equation}
As to the second term $T_j(\varphi;V)$ (integrals on $\Phi_j(\Omega)\setminus V$),
we exploit that $(\nabla \Psi_j)^{-1}$ is uniformly bounded and
$H$ has at most linear growth. Hence, there is a constant $C>0$ 
such that with $C_\varphi:=C \norm{\varphi}_{L^\infty(U)}$,
\begin{equation}\label{ubBC-14}
\begin{aligned}
	\limsup_{j\to\infty}\abs{T_j(\varphi;V)}\leq &\, C_\varphi \limsup_{j\to\infty} (\cL^N+|Dg_0|+|G_0|)(\Phi_j(\Omega)\setminus V) \\
	 \leq &\,  C_\varphi  (\cL^N+|Dg_0|+|G_0|)(\overline{\Omega} \setminus V) \\
	=&\, C_\varphi  (\cL^N+|Dg_0|+|G_0|)((\Omega\cup \Gamma) \setminus V),
\end{aligned}
\end{equation}
by dominated convergence and the fact that $\Phi_j\to \operatorname{id}$ in $C^1$.
Here, we also used that 
\[
	|Dg_0|(\partial\Omega\setminus \Gamma)=0=|G_0|(\partial\Omega\setminus \Gamma),
\]
by definition of $g_0$, $G_0$ and 
our assumption that
 $\cH^{N-1}(\overline{\Gamma}\setminus \Gamma)=0$.

For the term $S_j(\varphi;V)$ (integrals on $\Phi_j(\Omega)\cap V$) on the right hand side of \eqref{ubBC-13}, we again use that
$\Phi_j\to \operatorname{id}$ in $C^1$; in particular,  $(\nabla \Psi_j)^{-1}\to I$ (identity matrix) uniformly. In addition, 
$H\geq 0$ is Lipschitz and $\varphi$ is uniformly continuous. Consequently, for all $\varphi\geq 0$,
\begin{equation}\label{ubBC-15}
\begin{aligned}
	\int_{\overline{\Omega} \cap V} \!\! \varphi \,\de H(Dg_0,G_0)(x) \leq  \liminf_{j\to\infty} S_j(\varphi;V) \leq \limsup_{j\to\infty} S_j(\varphi;V) \leq \int_{\overline{\Omega}\cap V} \!\! \varphi \,\de H(Dg_0,G_0)(x).
\end{aligned}
\end{equation}
Here, to handle the limit in the domain of integration $\Phi_j(\Omega)\cap V$, for the lower bound we used monotonicity and the fact
that $\overline\Omega\subset \Phi_j(\Omega)$ for all $j$ by \eqref{ubBC-05},
while for the upper bound we used that $\Phi_j(\Omega)\searrow \overline{\Omega}$ 
and dominated convergence.

By splitting a general $\varphi$ into positive and negative parts,
\eqref{ubBC-15} immediately implies that 
\begin{equation}\label{ubBC-16}
\begin{aligned}
	\lim_{j\to\infty} S_j(\varphi;V) 
	= \int_{\overline{\Omega}\cap V} \varphi \,\de H(Dg_0,G_0)(x)
	=\int_{(\Omega\cup \Gamma)\cap V} \varphi \,\de H(Dg_0,G)(x)
\end{aligned}
\end{equation}
for all $\varphi\in C(\overline{U})$.
Combining \eqref{ubBC-13}, \eqref{ubBC-14} and \eqref{ubBC-16} 
for the case $V=\Omega\cup \Gamma$, 
we infer that
\begin{align}\label{ubBC-22}
	& \int_{\Omega} \varphi \,\de H(Dg_j,G_j)(z) \to
	\int_{\Omega\cup \Gamma} \varphi \,\de H(Dg_0,G)(x)\quad\text{as $j\to\infty$.}
\end{align}
In particular, \eqref{ubBC-22} yields \eqref{ubBC-11} when we choose $\varphi\equiv 1$.

In addition, we can analogously obtain 
\eqref{ubBC-22}
for other functions instead $H$
(globally Lipschitz with a uniform strong recession function in the sense of \eqref{str-uni-rec}; if needed, $H$ can be temporarily split into a positive and a negative part for the proof of \eqref{ubBC-16}, just like $\varphi$). 
With the choices
\[
	\text{$H(A,B):=A_{ij}$ and $H(A,B):=B_{ij}$, where $A=(A_{ij})$ and $B=(B_{ij})$,}
\]
for $i=1,\ldots,d$ and $j=1,\ldots,N$,
\eqref{ubBC-22}
implies the second part of \eqref{ubBC-10}, in particular that 
$
	Dg_j\res{\Omega} \,\wsto\, Dg_0\res(\Omega\cup \Gamma)
	~\text{in $\cM(\overline\Omega;\R{d\times N})$}.
$
Finally, it is not hard to see that 
$g_j\rightarrow g$ in $L^1(\Omega;\R{d})$. 
We conclude that $g_j\wsto g$ in $BV(\Omega;\R{d})$, which completes the proof of \eqref{ubBC-10}.

{\bf Step 2: Approximating limit states $(\hat{g}_j,\hat{G}_j)$ with $\hat{g}_j=u_0$ on $\Gamma$.}

The functions $g_j$ defined in the previous step do not yet satisfy 
$g_j=u_0$ on $\Gamma$, although their traces converge to $u_0$ by \eqref{ubBC-09}.
We can correct this using the trace extension theorem:
Choose $\{v_j\}\subset W^{1,1}(\Omega;\R{d})$ such that
\begin{equation}\label{ubBC-50}
\begin{aligned}
	T_\Omega v_j|_\Gamma=T_\Omega (g_j-u_0)|_\Gamma
	\quad\text{and}\quad
	\norm{v_j}_{W^{1,1}(\Omega;\R{d})}\leq C_{\partial\Omega} 
	\norm{T_\Omega (g_j-u_0)}_{L^1(\Gamma;\R{d})}.
\end{aligned}
\end{equation}
By \eqref{ubBC-09}, we infer that $\norm{v_j}_{W^{1,1}(\Omega;\R{d})}\to 0$.
Consequently, for
\[
	\hat{g}_j:=g_j-v_j\quad\text{and}\quad \hat{G}_j\coloneqq G_j,
\]
instead of $(g_j,G_j)$
we still have \eqref{ubBC-09}, \eqref{ubBC-10} and \eqref{ubBC-11}, and in addition, $\hat{g}_j=u_0$on $\Gamma$. Namely,
defining 
\[\Theta\coloneqq  (T_\Omega g- u_0)\otimes \nu_{\Gamma} \mathcal H^{N-1}\res{\Gamma}
\]
so that $Dg_0= Dg + \Theta$ on $\Omega \cup \Gamma$, we have that
\begin{equation}\label{ubBC-09*}
	\hat g_j =  u_0 \hbox{ on }\Gamma, 
	\quad
	|\hat G_j|(\Gamma)=|\hat G_j|(\partial\Omega)=0,
\end{equation}
\begin{equation}\label{ubBC-10*}
	\begin{aligned}
		\hat g_j|_{\Omega} \wsto g
	~~~\text{in}~BV(\Omega;\R{d}), \quad
	 (D\hat g_j\res{\Omega},\hat G_j\res{\Omega}) \,\wsto\, (Dg+ \Theta,G)
	~~~\text{in $\cM(\Omega\cup \Gamma;\R{d\times N})^2$},
\end{aligned}
\end{equation}
and
\begin{equation}\label{ubBC-11*}
\begin{aligned}
	 \lim_{j \to \infty }\int_\Omega \,\de H(D\hat g_j,\hat G_j)=\int_\Omega \,\de H(g,G)+ \int_{\Gamma} \,\de H^\infty(\Theta, G).
\end{aligned}\end{equation}

{\bf Step 3: Recovery by diagonalizing free recovery sequences for $(\hat{g}_j,\hat{G}_j)$}

We first observe that $I$ in \eqref{defI} admits the following equivalent representation
\begin{equation}\label{defI_g}
	I_g(g,G;\Omega)\coloneqq \inf\Big\{\liminf_{n\to\infty} E(u_n;\Omega): \{u_n\}\subset SBV(\Omega;\R{d}), u_n \wsto (g,G), u_n \equiv g \hbox{ on }\partial \Omega\Big\},
\end{equation}
for every $(g,G) \in mSD$.

Clearly $I(g,G;\Omega)\leq I_g(g,G;\Omega)$. The opposite one can be obtained following 
an argument of \cite{BouFoLeMa02a}. The details are provided below for the reader's convenience.

For any $ SBV(\Omega;\mathbb R^d) \ni u_n \wsto(g, G)$ in the sense of~\eqref{006}, almost optimal for $I(g,G;\Omega)$, i.e., for every $\varepsilon >0$, 
\[\liminf_{n\to \infty}E(u_n;\Omega)\leq I(g,G;\Omega)+ \varepsilon.\]
Without loss of generality, assume that the above lower limit is indeed a limit
and consider the sequence of measures $\nu_n\coloneqq  \mathcal L^N+ |Du_n|+ |D g|$, which converges weakly* to some Radon measure~$\nu$.

Denoting, for every $t>0$, $\Omega_t \coloneqq \{x\in \Omega |\ \mathrm{dist} (x,\partial \Omega) > t\}$, we fix some $\eta >0$ and for every $0<\delta <\eta$ we define the
subsets $L_\delta \coloneqq \Omega_{\eta-2\delta}\setminus\overline{ \Omega_{\eta+\delta}}$.
Consider a smooth cut-off function $\varphi_\delta\in\mathcal{C}_0^\infty(\Omega_{\eta-\delta};[0,1])$ such that $\varphi_\delta=1$ on $\Omega_{\eta}$. 
As the thickness of the strip $L_\delta$ is of order $\delta$, we have an upper bound of the form $\|\nabla \varphi_\delta\|_{L^\infty (\Omega_{\eta-\delta})}\le
C/\delta$. 
Define
\begin{equation*}
	w^\delta_n \coloneqq u_n\varphi_\delta + g (1-\varphi_\delta).
\end{equation*}
Clearly this sequence converges to $g$ in $L^1 (\Omega;\mathbb{R}^d)$ and satisfies $T_\Omega w_n^\delta=T_\Omega g$ on $\partial \Omega$.

Moreover
$$
\nabla w_n^\delta \wsto G 
\quad\text{and}\quad D w_n^\delta \wsto D g  \hbox{ in } \mathcal M(\Omega;\mathbb R^{d \times N})
$$
as $n \to\infty$ and then $\delta \to 0$.
Indeed
\begin{align*}
	\nabla w_n^\delta = \nabla \varphi_\delta \otimes (u_n-g)+ \varphi_\delta (\nabla u_n- \nabla g)+ \nabla g, \\
	D w_n^\delta = \nabla \varphi_\delta \otimes (u_n-g)+ \varphi_\delta (D u_n- D g)+ D g.
\end{align*}

Concerning the energies, we have
\begin{align*}
	E(w_{n}^\delta; \Omega)  &\le  E(w_{n}^\delta; \Omega_{\eta})+E(w_{n}^\delta; \Omega\setminus%
	\overline{\Omega_{\eta-\delta}})+ E(u_{n}^\delta; \Omega_{\eta-2\delta}\setminus\overline{%
		\Omega_{\eta+\delta}}) \\
	& \le \begin{aligned}[t]
		& E(u_{n}; \Omega_{\eta}) + E(g; \Omega\setminus \overline{\Omega_{\eta-\delta}}) \\
		&+ C_{W, \psi}
	\left((\mathcal L^N+ |Du_{n}|+ |Dg|)(L_\delta) 
	+	\frac1{\delta}\int_{L_\delta} |u_n - g| \,\de x\right),
	\end{aligned}
\end{align*}
where $C_{W,\psi}$ is any bigger constant which bounds from above the constants appearing in \eqref{W1}, \eqref{psi1} and in $L^\infty$ bound of $\nabla \psi_\delta$ on $L_\delta$.
Taking the limit as $n\to \infty$ we have
\begin{align*}
	\liminf_{n\to\infty} E(w_{n}^\delta;\Omega)&\le
	\lim_{n\to\infty}E(u_{n};\Omega)+ C_{W,\psi}\nu(\Omega\setminus \overline{\Omega_{\eta-\delta}}%
	)+C_{W,\psi}\nu(\overline{L_\delta}) \\
	&\le I(g,G;\Omega)+\varepsilon+ C_{W,\psi}\nu(u; \Omega\setminus \overline{%
		\Omega_{\eta-\delta}})+C_{W,\psi}\nu(\overline{L_\delta}). 
\end{align*}
Letting $\delta\to0$ we obtain
\begin{equation*}
	I_g(g,G;\Omega)\le I(g,G;\Omega)+\varepsilon+C_{W,\psi}\nu(\Omega\setminus
	\Omega_{\eta})+C_{W,\psi}\nu(\partial \Omega_\eta).
\end{equation*}
Choose a subsequence $\{\eta_n\}$ such that $\eta_n\to 0^+$ and $%
\nu(\partial A_{\eta_n})=0$. By letting first $n\to\infty$ and then $%
\varepsilon\to 0^+$ we conclude that $I_g(g,G;\Omega)\leq I(g,G;\Omega)$.

Then, for any $(\hat{g}_j,\hat{G}_j) \wsto (g,G)$ as in Step 2, satisfying \eqref{ubBC-09*}, \eqref{ubBC-10*}, and \eqref{ubBC-11*}, we can apply Theorem~\ref{thm:representation} and find a recovery sequence for $I_{\hat{g}_j}(\hat{g}_j,\hat{G}_j;\Omega)=I(\hat{g}_j,\hat{G}_j;\Omega)$ for each $j$,
i.e., $\{u_n^j\}_n\subset SBV(\Omega;\mathbb R^d)$ such that
$ u_n^j \wsto (\hat g_j, \hat G_j)$ in the sense of~\eqref{006}, $T_\Omega u_n^j= T_\Omega \hat g_j$ on $\partial \Omega$, in particular $T_\Omega u_n^j= u_0$ on $\Gamma$, and 
$$
\lim_{n \to \infty} E(u_n;\Omega)= J(g_j,G_j;\Omega)=\int_\Omega \,\de H(\hat g_j,\hat G_j). 
$$
Since $|\hat G_j|(\partial \Omega)=0$, we also have that
$\nabla u_n^j \wsto \hat G_j $ in $\mathcal M(\Omega \cup \Gamma;\mathbb R^{d \times N})$.

A standard diagonalization argument, exploiting the coercivity of $E$ given by \eqref{W1} and \eqref{psi1} to obtain bounds uniform in $n$ and $j$, now concludes the proof.
\end{proof}

\section{Further properties and examples}\label{sec:further}

As shown below, we also have an alternative way of interpreting $I$, as a more classic relaxation problem of a functional 
on $SBV\times L^1$ in $BV\times \cM$.
\begin{theorem}\label{thm:I-alternative}
Assume \eqref{W1}--\eqref{W3} and \eqref{psi1}--\eqref{psi3}.
For $(g,G)\in SBV(\Omega;\R{d})\times L^1(\Omega;\R{d\times N})$ and $R>0$, we define
\begin{equation}\label{E_R}
	\hat{E}_R(g,G;\Omega)\coloneqq \int_\Omega (W(\nabla g)+R|\nabla g-G|)\,\de x+
	\int_{S_g\cap\Omega}\psi([g],\nu_g)\,\de \cH^{N-1}(x)
\end{equation}
and its relaxation
\[
	\hat{I}_R(g,G;\Omega)\coloneqq 
	 \inf\Big\{\liminf_{n\to\infty} 
	\hat{E}_R(g_n,G_n;\Omega) \,\Big|\,
	SBV\times L^1 \ni
	(g_n,G_n) \wsto (g,G)\text{ in }BV\times \cM\Big\}
\]
for $(g,G)\in 
mSD$.
Then there exists $R_0=R_0(N,W,\psi)>0$ 
such that 
\[
	\hat{I}_R(\cdot,\cdot;\Omega)=I(\cdot,\cdot;\Omega)\quad\text{for all $R\geq R_0$}, 
\]
where $I(\cdot,\cdot;\Omega)$ is the relaxation of $E(\cdot;\Omega)$ defined in
\eqref{defI}.
\end{theorem}
\begin{remark}\label{rem:I-alternative}
Theorem~\ref{thm:I-alternative} in principle opens another route to proving Theorem~\ref{thm:representation}, our representation theorem for~$I$, via a relaxation theorem characterizing~$\hat{I}_R$. 
However, the closest available results in this direction seem to be
\cite{ArDPRi20a,BaCheMaSa13a} (for the case $\cA=(\operatorname{Curl},0)$)
and \cite{BaBouBuFo96a}. However, the former does not allow us to choose $\psi$ freely, and the latter
does not allow us to include $G$.
\end{remark}
\begin{proof}[Proof of Theorem~\ref{thm:I-alternative}]
We first observe that, for every $R >0$, $\hat E_R(u, \nabla u, \Omega)= E(u;\Omega)$ for every $u \in SBV(\Omega;\mathbb R^d)$.
Let $(g, G) \in mSD$ and let $SBV(\Omega;\mathbb R^d) \ni g_n \wsto(g, G)$ according to \eqref{SBV_conv}.
Since $\{(g_n, \nabla g_n)\}$ is an admissible sequence for $\hat E_R(g, G;\Omega)$,  
\[\hat I_R(g,G;\Omega)\leq \liminf_{n\to \infty} E(g_n;\Omega).
\]  
Hence, passing to the infimum over all the admissible sequences $\{g_n\}$, we have 
\[\hat I_R(g,G,\Omega)\leq I(g,G, \Omega).\]
To prove the opposite inequality for $R \geq R_0$ with a suitable $R_0$ to be chosen later, take $\{(g_n, G_n)\}$ admissible for $\hat I_R(g,G;\Omega)$, so that $g_n \wsto g$ in $BV$, $G_n \wsto G$ in $\mathcal M$.
We choose a sequence $\{v_n\}$ given by \cite[Theorem~1.1]{Silhavy2015} 
such that
\[
	v_n \wsto 0\quad\text{in $BV$,} \qquad \nabla v_n= -\nabla g_n +G_n\, , 
\]
and
\begin{align}\label{estapp}|D v_n|(\Omega) \leq C(N) \int_\Omega|G_n-\nabla g_n| \, \de x = C(N) \int_\Omega|\nabla v_n| \, \de x.
\end{align}
In particular, the sequence $u_n \coloneqq g_n+ v_n$ 
is admissible for $I(g, G;\Omega)$.

Taking into account  that 
\[
	S_{g_n}= (S_{g_n} \setminus S_{g_n+ v_n}) \cup (S_{g_n}\cap S_{g_n+ v_n})
	\quad\text{and}\quad
	S_{g_n+ v_n}=   (S_{g_n}\cap S_{g_n+ v_n}) \cup (S_{g_n+ v_n}\setminus S_{g_n}), 
\]	
also using \eqref{W1}, \eqref{psi1} and \eqref{psiLip} we obtain that
\begin{align*}
\begin{aligned}
&\hat E_R (g_n, G_n; \Omega ) - \hat E_R (g_n+v_n, \nabla (g_n+ v_n);\Omega) \\
=&
	\begin{aligned}[t]
		&\int_\Omega \big(W(\nabla g_n)- W(\nabla g_n+ \nabla v_n)\big)\; \de x 
			+ \int_\Omega R |\nabla g_n- G_n|\; \de x  \\
		&+ \int_{\Omega \cap S_{g_n}} \psi([g_n], \nu_{g_n}) \; \de \mathcal H^{N-1}(x)
			- \int_{\Omega \cap S_{g_n+ v_n}} \psi([g_n+ v_n], \nu_{g_n+v_n}) \; \de \mathcal H^{N-1}(x)
	\end{aligned}\\		
\geq&
	\begin{aligned}[t]		
		&- \int_{\Omega}L |\nabla v_n|\;  \de x + \int_{\Omega} R |\nabla v_n| \de x 
			+ \int_{\Omega \cap (S_{g_n} \setminus (S_{g_n+ v_n})} \psi([g_n], \nu_{g_n}) \de \mathcal H^{N-1}(x)  \\
		&- \int_{\Omega \cap (S_{g_n} \cap S_{g_n+ v_n})} C_{\psi} |[v_n]| \de \mathcal H^{N-1}(x)
			- \int_{\Omega \cap (S_{g_n+v_n} \setminus S_{g_n})} C_{\psi} |[v_n]| \de \mathcal H^{N-1}(x)
	\end{aligned}\\
\geq&  \int_{\Omega} (R-L)|\nabla v_n| \de x - \int_{\Omega \cap S_{v_n}} C_\psi |[v_n]|\, \de \mathcal H^{N-1}(x)\\
\geq& \, (R- (L+ C_\psi C(N)))\int_{\Omega} |\nabla v_n|\, \de x 
		\geq 0, 
\end{aligned}
\end{align*}
as long as $R\geq R_0\coloneqq L+ C_\psi C(N)$.
Here, $L$, $C_\psi$, and $C(N)$ denote the Lipschitz constant of $W$, the Lipschitz and growth constant of $\psi$ in \eqref{psiLip} and \eqref{psi1}, and the constant appearing in~\eqref{estapp}, respectively.
Passing to the limit as $n\to \infty$, we conclude that
\begin{align*}
\begin{aligned}
\liminf_{n\to \infty} \hat E_R (g_n, G_n; \Omega ) 
	&\geq \liminf_{n\to \infty} \hat E_R (g_n+v_n, \nabla (g_n+ v_n);\Omega)
	=\liminf_{n\to\infty} E(g_n+v_n;\Omega)
	\geq I(g, G;\Omega)
\end{aligned}
\end{align*}	
for all $R\geq R_0$.
As this holds for all sequences $\{(g_n,G_n)\}$ that are admissible for $\hat I_R(g,G;\Omega)$, the thesis follows.
\end{proof}

In view of Theorem~\ref{thm:I-alternative}, it is a natural question to what degree 
our relaxed functional~$I$ is influenced 
by its origin from~$E$, defined on structured deformations. 
The following example shows that this special background is still present in the relaxed $I$ at least in the sense
that not all quasiconvex-convex densities~$H$ 
(that could be obtained by general relaxation in $BV\times \cM$) can be obtained in~$I$.
\begin{proposition}\label{counterexample}
For all $W$ and $\psi$ satisfying
the assumptions of Theorem~\ref{thm:representation},
there exists $B_0\in \R{d\times N}$ and $\xi\in \R{d}$, $\nu\in \R{N}$ with $\abs{\xi}=\abs{\nu}=1$ so that for the function $H$ defined in \eqref{def-H},
\begin{equation}\label{ex-H-repr}
	H(B_0+t\xi\otimes \nu,B_0)=W(B_0)+\psi(t\xi,\nu)\quad\text{for all $t>0$}.
\end{equation}
In particular, for any possible choice of $W$ and $\psi$,
\[
	H\neq H_0\quad\text{with}\quad H_0(A,B)\coloneqq\sqrt{|A|^2+1}+|B|
\] 
because the function $(0,+\infty)\ni t\mapsto H_0(B_0+t\xi\otimes \nu,B_0)$ 
is not affine.
\end{proposition}
\begin{proof}
To see ``$\leq$'' in \eqref{ex-H-repr}, it suffices to choose a suitable admissible sequence in \eqref{defH-seq}, the 
sequential characterization of $H$: On $Q$, we have $\nabla u_n\wsto B_0$ in $\cM$ and 
$u_n\wsto A_0 x$ for $A_0\coloneqq B_0+t\xi\otimes \nu$ for
\[
	u_n(x)\coloneqq B_0 x + \frac{t\xi}{n} [nx\cdot \nu],
\]
where $[s]\coloneqq\min\{z\in \Z{}:|z-s|=\dist{s}{\Z{}}\}$ denotes rounding of $s$ to the closest integer.
An upper bound for $H(A_0,B_0)$ is therefore given by $\lim_n E(u_n;Q)=W(B_0)+\psi(t\xi,\nu)$ (by $1$-homogeneity of $\psi$),
for any possible choice of $B_0$, $\xi$ and $\nu$.

To obtain ``$\geq$'' in \eqref{ex-H-repr}, we use a particular choice:
Since both $W$ and $\psi$ are continuous and $W$ is coercive, there always exists 
global minima $B_0$ of $W$ on $\R{d\times N}$ and $(\xi,\nu)$ of $\psi$ on the compact set
$\S{d-1}\times \S{N-1}\coloneqq\big\{(\xi,\nu)\in \R{d}\times \R{N}\mid \abs{\xi}=\abs{\nu}=1\big\}$. As a consequence,
\begin{equation}\label{ex-choices}
	W(B_0)=W^{**}(B_0)~~~\text{and}~~~\psi(\xi,\nu)\leq \psi(\tilde{\xi},\tilde{\nu})
	~~\text{for all $(\tilde{\xi},\tilde{\nu})\in \S{d-1}\times \S{N-1}$}.
\end{equation} 
Here, $W^{**}$ denotes the convex hull of $W$.
For any $u$ admissible in the definition \eqref{def-H} of $H(A_0,B_0)$ with $A_0\coloneqq B_0+t\xi\otimes \nu$,
we now have that
\begin{equation}\label{ex-Wbelow}
	\int_Q W(\nabla u)\,\de x\geq \int_Q W^{**}(\nabla u)\,\de x \geq W^{**}(B_0)=W(B_0)
\end{equation}
by Jensen's inequality, since $\ave_Q \nabla u\,\de x=B_0$ for all admissible $u$ in \eqref{def-H}.
Moreover, since $u\in SBV$ and $u=A_0x$ on $\partial Q$, we have that 
\begin{equation*}
	\int_{S_u\cap Q} [u]\otimes \nu_u \,\de \cH^{N-1}(x)=\int_Q \,\de Du-\int_Q \nabla u\,\de x=A_0-B_0=
	t\xi\otimes \nu.
\end{equation*}
Multiplied with the fixed unit vector $\nu$ from the right, this reduces to
\begin{equation}\label{ex-choices-xi}
	\int_{S_u\cap Q} [u] (\nu_u\cdot \nu) \,\de \cH^{N-1}(x)=t\xi.
\end{equation}
By the positive $1$-homogeneity of $\psi$, the minimality property of $(\xi,\nu)$ in \eqref{ex-choices} and
another application of Jensen's inequality 
with \eqref{ex-choices-xi} to the convex function $|\cdot|$, we infer that
\begin{equation}\label{ex-psibelow}
\begin{aligned}
	\int_{S_u\cap Q} \psi([u],\nu_u) \,\de\cH^{N-1}(x) &=
	\int_{S_u\cap Q} |[u]| \psi\Big(\frac{[u]}{|[u]|},\nu_u\Big) \,\de\cH^{N-1}(x)\\
	&\geq \int_{S_u\cap Q} |[u](\nu_u\cdot \nu)| \psi(\xi,\nu)\,\de\cH^{N-1}(x)\\
	&\geq |t\xi| \psi(\xi,\nu)=\psi(t\xi,\nu)
\end{aligned}
\end{equation}
for all $t>0$.
Combining \eqref{ex-Wbelow} and \eqref{ex-psibelow}, we conclude that $E(u;Q)\geq W(B_0)+\psi(t\xi,\nu)$ for all $u$ admissible in 
\eqref{def-H} with $(A,B)=(A_0,B_0)$. This implies the asserted lower bound for $H(A_0,B_0)$.
\end{proof}

\noindent\textbf{Acknowledgements} 
The authors are grateful to David R.~Owen for inspiring discussions on this topic.
Part of this work has been developed during the visits at UTIA in Prague by MM and EZ who are very grateful for the hospitality and the support. 
SK acknowledges the support of INdAM-GNAMPA through the project \emph{Programma Professori Visitatori 2023}, CUP E53C22001930001, and of Dipartimento di Scienze di Base ed Applicate per l'Ingegneria di Sapienza, Università di Roma.
MK, SK, and EZ are grateful to Dipartimento di Matematica at Politecnico di Torino, where also part of this work has been undertaken.
SK and MK were supported by the GA\v{C}R project 21-06569K. 
The research of EZ is part of the projects GNAMPA-INdAM 2023 \emph{Prospettive nelle scienze dei materiali: modelli variazionali, analisi asintotica e omogeneizzazione} CUP E53C22001930001, and PRIN \emph{Mathematical Modeling of Heterogeneous Systems} CUP 853D23009360006. 
MM and EZ are members of GNAMPA (INdAM). 
MM acknowledges support from the PRIN2020 grant \emph{Mathematics for Industry 4.0} (2020F3NCPX). 
This study was carried out within the \emph{Geometric-Analytic Methods for PDEs and Applications} project (2022SLTHCE), funded by European Union -- Next Generation EU within the PRIN 2022 program (D.D.~104 - 02/02/2022 Ministero dell’Università e della Ricerca). This manuscript reflects only the authors’ views and opinions and the Ministry cannot be considered responsible for them.

\bibliography{bib} 
\bibliographystyle{abbrv}

\end{document}